\documentclass{article}
\usepackage{amsrefs}
\usepackage{amsmath}
\usepackage{amsfonts}
\usepackage{amsthm}
\usepackage{geometry}
\usepackage{enumerate}
\ifdefined\directlua
\usepackage{polyglossia}
\setmainlanguage{english}
\usepackage{breqn}
\usepackage{fontspec}
\usepackage{microtype}
\else
\usepackage[british]{babel}
\fi
\newcommand{\FF}{\mathbb F}

\newcommand{\RR}{\mathbb R}
\newcommand{\KK}{\mathbb K}

\newcommand{\fT}{\mathfrak T}

\newcommand{\cG}{\mathcal G}

\newcommand{\cS}{\mathcal S}
\newcommand{\cQ}{\mathcal Q}

\newcommand{\cP}{\mathcal P}

\newcommand{\GG}{\mathbb G}

\newcommand{\cW}{\mathcal W}
\newcommand{\Rad}{\mathrm{Rad}\,}
\newcommand{\chr}{\mathrm{char}}

\newcommand{\PG}{\mathrm{PG}}

\newcommand{\GL}{\mathrm{GL}}

\newcommand{\rk}{\mathrm{rk}}
\newcommand{\cod}{\operatorname{cod}}

\newtheorem{problem}{Problem}
\theoremstyle{theorem}
\newtheorem{theorem}{Theorem}[section]
\newtheorem{lemma}[theorem]{Lemma}
\newtheorem{corollary}[theorem]{Corollary}
\newtheorem{prop}[theorem]{Proposition}
\theoremstyle{definition}

\theoremstyle{definition}
\newtheorem{note}{Note}

\title{A geometric approach to alternating $k$-linear forms}
\author{Ilaria Cardinali, Luca Giuzzi and Antonio Pasini}
\begin{document}
\maketitle
\begin{abstract}
Given an $n$-dimensional vector space $V$ over a field $\KK$, let $2\leq k < n$. A natural one-to-one correspondence exists between the alternating $k$-linear forms of $V$ and the linear functionals of $\bigwedge^kV$, an alternating $k$-linear form $\varphi$ and a linear functional $f$ being matched in this correspondence precisely when $\varphi(x_1,\ldots, x_k) = f(x_1\wedge\cdots\wedge x_k)$ for all $x_1,\ldots, x_k \in V$. Let $\varepsilon_k:\cG_k(V)\rightarrow \PG(\bigwedge^kV)$ be the Pl\"{u}cker embedding of the $k$-Grassmannian $\cG_k(V)$ of $V$. Then $\varepsilon_k^{-1}(\ker(f)\cap\varepsilon_k(\cG_k(V)))$ is a hyperplane of the point-line geometry $\cG_k(V)$. It is well known that all hyperplanes of $\cG_k(V)$ can be obtained in this way, namely every hyperplane of $\cG_k(V)$ is the family of $k$-subspaces of $V$ where a given alternating $k$-linear form identically vanishes. For a hyperplane $H$ of $\cG_k(V)$, let $R^\uparrow(H)$ be the subset (in fact a subspace) of $\cG_{k-1}(V)$ formed by the $(k-1)$-subspaces $A\subset V$ such that $H$ contains all $k$-subspaces that contain $A$. In other words, if $\varphi$ is the (unique modulo a scalar) alternating $k$-linear form defining $H$, then the elements of $R^\uparrow(H)$ are the $(k-1)$-subspaces $A = \langle a_1,\ldots, a_{k-1}\rangle$ of $V$ such that $\varphi(a_1,\ldots, a_{k-1},x) = 0$ for all $x\in V$. In principle, when $n-k$ is even it might happen that $R^\uparrow(H) = \emptyset$. When $n-k$ is odd then $R^\uparrow(H) \neq \emptyset$, since every $(k-2)$-subspace of $V$ is contained in at least one member of $R^\uparrow(H)$, but it can happen that every $(k-2)$-subspace of $V$ is contained in precisely one member of $R^\uparrow(H)$. If this is the case, we say that $R^\uparrow(H)$ is
 \emph{spread-like}. In this paper we obtain some results on $R^\uparrow(H)$ which answer some open questions from the literature  and suggest the conjecture that, if $n-k$ is even and at least $4$, then $R^\uparrow(H) \neq \emptyset$ but for one exception with
$\KK\leq\RR$ and $(n,k) = (7,3)$, while if $n-k$ is odd and at least $5$ then $R^\uparrow(H)$ is never spread-like.
\end{abstract}
\noindent
{\bfseries Keywords:} Grassmann Geometry; Hyperplane; Multilinear form; Alternating form.
\par\noindent
{\bfseries MSC:} 15A75; 14M15; 15A69.

\section{Introduction}

\subsection{Embeddable Grassmannians and their hyperplanes}\label{introduction}

Let $V$ be an $n$-dimensional vector space over a field $\KK$. For $1\leq k < n$ denote by ${\cG}_k(V)$ the $k$-Grassmannian of $V$, namely the point-line geometry whose points are the $k$--dimensional vector subspaces of $V$ and whose lines are the sets $\ell_{Y,Z}:=\{X~\colon ~Y\subset X\subset Z, ~\dim X=k\}$ where $Y$ and $Z$ are subspaces of $V$ with $Y\subset Z,$ $\dim Y=k-1$
and $\dim Z=k+1$. Incidence is containment. In particular, $\cG_1(V) = \PG(V)$ and $\cG_{n-1}(V) \cong \cG_1(V^*)$, where $V^*$ is the dual of $V$.

Using the Pl\"{u}cker embedding $\varepsilon_k\colon \cG_k(V) \rightarrow \PG(\bigwedge^kV)$, which maps every $k$-subspace $\langle v_1,\ldots, v_k\rangle$ of $V$ onto the point $\langle v_1\wedge\cdots\wedge v_k\rangle$ of $\PG(\bigwedge^kV)$, the point-set of $\cG_k(V)$ is mapped onto a projective variety $\GG_k(V) \subset \PG(\bigwedge^k V)$. It is well known that $\GG_k(V)$ spans $\PG(\bigwedge^kV)$.

According to the terminology commonly used for point-line geometries, a subspace of $\cG_k(V)$ is a set $S$ of points of $\cG_k(V)$ such that if a line $\ell$ of $\cG_k(V)$ meets $S$ in at least two distinct points then $\ell\subseteq S$. A hyperplane of $\cG_k(V)$ is a proper subspace of $\cG_k(V)$ which meets
every line of $\cG_k(V)$ non-trivially. In view of a theorem of Shult \cite{shult92},
every hyperplane $H$ of $\cG_k(V)$ arises as a hyperplane section
from $\varepsilon_k$, namely $H=H_f$ where $H_f =\varepsilon_k^{-1}(\ker(f)\cap \GG_k(V))$ for a non-null linear functional $f$
on $\bigwedge^kV$. Equivalently, if $\alpha_f:V\times\ldots\times V\rightarrow \KK$ is the alternating $k$-linear form associated to $f$, defined by the clause $\alpha_f(v_1,\ldots, v_k) := f(v_1\wedge\cdots\wedge v_k)$, the hyperplane $H$ is the set of $k$-subspaces of $V$ where $\alpha_f$ identically vanishes. So, the hyperplanes of $\cG_k(V)$ bijectively correspond to proportionality classes of non-zero linear functionals of $\bigwedge^k V$
 as well as proportionality classes of non-trivial $k$-linear alternating forms of $V$.

Clearly, when $k = 1$ or $k = n-1$ the previous definition gives us back the hyperplanes (respectively, dual hyperplanes) of $\PG(V)$, defined as usual.
Suppose $k = 2$ and take $H$ as a hyperplane of $\cG_2(V)$. Let ${\cS}(H)$ be the point-line geometry defined as follows: its points are just the points of $\PG(V)$ and the lines are the members of $H$, regarded as lines of $\PG(V)$. The following proposition is crucial for all we are going to say in this paper. It has to do with polar spaces. We presume that the reader is fairly familiar with them; if not, we refer him$/$her to Buekenhout and Cohen \cite[Chapters 7-10]{BC} or Shult \cite[Chapter 7]{shult11}.

\begin{prop}\label{Polar Space}
The geometry ${\cS}(H)$ is a (possibly degenerate) polar space of symplectic type.
\end{prop}
\begin{proof}
This statement immediately follows from what we have said above on hyperplanes of $\cG_k(V)$ and alternating $k$-linear forms. However, it can also be proved by the following straightforward argument, with no recourse to \cite{shult92}. Actually, in this way we obtain anew the result of \cite{shult92} for
the special case $k = 2$.

Given a line $L\in H$ of ${\cS}(H)$ and a point $p\not \in L$, let $X = \langle L, p\rangle$; this is a plane of $\PG(V)$.
The set $\ell_{p,X}$ of all lines of $X$ through $p$ is a line of $\cG_2(V)$. As $H$ is a hyperplane of $\cG_2(V)$, either $\ell_{p,X}\subset H$ or $\ell_{p,X}$ contains just one element $M\in H$. Thus, $p$ is collinear in ${\cS}(H)$ with either all the points of $L$ or just one of them; hence ${\cS}(H)$ is a polar space.
 Moreover, by definition, ${\cS}(H)$ is embedded in $\PG(V)$ in such a way that all points of $\PG(V)$ are points of ${\cS}(H)$. It follows that ${\cS}(H)$ is associated to an alternating bilinear form (see \cite[Chapter 9]{BC}).
\end{proof}
\begin{note}
  Although in this paper we are only interested in embeddable Grassmannians,
  it is worth mentioning that if $\cG_2(V)$ were non-embeddable, namely $\KK$ were a non-commutative division ring rather than a field, then the above argument would imply that every hyperplane of $\cG_2(V)$ is trivial in the sense of Section~\ref{trivial} of this paper, namely that it contains precisely those lines of $\PG(V)$ that meet a given subspace of codimension $2$ non-trivially. This fact can be used as the initial step of an inductive proof that all hyperplanes of a non-embeddable Grassmannian are trivial, thus
proving anew the result of~\cite{hall-shult93}.
\end{note}

\subsection{The problems studied in this paper}\label{problems}

Still under the assumption $k = 2$ and taking $H$ as above, let $R(H)$ be the radical of the polar space ${\cS}(H)$. Namely, $R(H)$ is the subspace of $\PG(V)$ formed by the points $p\in \PG(V)$ such that all lines of $\PG(V)$ through $p$ belong to $H$. Clearly, nothing can be said on $R(H)$ in general, except that $R(H)$ has even codimension in $\PG(V)$, since ${\cS}(H)$ is of symplectic type. In particular, when $n$ is even it can happen that $R(H) = \emptyset$ and, when $n$ is odd, $R(H)$ might consist of a single point.

It is natural to ask what are the structures that deserve to be taken as the analogues of $R(H)$ when $k > 2$ and investigate what can be said about them
in general. So, let $k \geq 2$ and let $H$ be a hyperplane of $\cG_k(V)$. Given a subspace $X$ of $V$ of dimension $\dim(X) < k$, let $(X)G_k$ be the set of $k$-subspaces of $V$ that contain $X$. For $1 \leq i < k$ let $R_i(H)$ be the set of $i$-subspaces $X$ of $V$ such that $(X)G_k\subseteq H$. As we shall prove later (Proposition~\ref{V2(H) subspace}), the set $R_i(H)$ is a subspace of $\cG_i(V)$. It is natural to regard this subspace as the $i$-\emph{radical} of $H$. However, the programme of investigating $i$-radicals for any $i < k$ is perhaps to broad to be feasible at this stage.
Thus, we shall here  consider only $R_{k-1}(H)$ and $R_1(H)$. We put
\[R^\uparrow(H) := R_{k-1}(H), ~~~ R_\downarrow(H) := R_1(H)\]
and we respectively call them the \emph{upper} and \emph{lower radical} of $H$.
When $k = 2$ the lower and the upper radical coincide and are equal to the radical $R(H)$ of the polar space ${\cS}(H)$. In this case there is nothing new to say. When $k \geq 3$ things become more interesting.

Suppose now $k \geq 3$. If we regard $H$ as the set of $k$-subspaces of $V$ where a given non-trivial alternating $k$-linear form $\alpha$ identically vanishes, then $R_\downarrow(H)$ is just (the subspace of $\PG(V)$ corresponding to) the radical $\Rad(\alpha)$ of $\alpha$, that is
\[ \Rad(\alpha) ~:= ~ \{v\in V ~:~ \alpha(x_1,\ldots, x_{k-1},v) = 0, ~ \forall x_1,\ldots, x_{k-1} \in V\}.\]
Clearly, $R_\downarrow(H)$ is a proper subspace of $\PG(V)$ (otherwise $\alpha$ would be trivial). In fact, we shall prove later (Proposition~\ref{codimension lower radical}) that $R_\downarrow(H)$ has codimension at least $k$ in $\PG(V)$.

The upper radical $R^\uparrow(H)$ looks more intriguing than $R_\downarrow(H)$. It contains all $(k-1)$-dimensional subspaces of $V$ that meet $R_\downarrow(H)$ non-trivially, but, in general, it contains far more elements than just these.
Actually, $R^\uparrow(H)$ can be quite large even when $R_\downarrow(H)$ is small or even empty. Henceforth we shall focus our attention on $R^\uparrow(H)$.

We firstly state a few more definitions which will be useful in view of our investigation of $R^\uparrow(H)$. For a $(k-2)$-subspace $X$ of $V$, the set $(X)G_k$ of $k$-subspaces of $V$ containing $X$ is a subspace of $\cG_k(V)$. Let $(X)\cG_k$ be the geometry induced by $\cG_k(V)$ on $(X)G_k$ and put $(X)H := (X)G_k\cap H$. Then $(X)\cG_k \cong \cG_2(V/X)$ and either $(X)H = (X)G_k$ or $(X)H$ is a hyperplane of $(X)\cG_k$. In either case, by Proposition~\ref{Polar Space}, the point-line geometry ${\cS}_X(H) = ((X)G_{k-1}, (X)H)$ is a polar space of symplectic type (possibly a trivial one, when $(X)H = (X)G_k)$). Let $R_X(H) := \Rad({\cS}_X(H))$ be the radical of ${\cS}_X(H)$. The following is straightforward.
\begin{prop}\label{polar space bis}
$R^\uparrow(H)\cap (X)G_{k-1} ~ = ~ R_X(H)$ for every $(k-2)$-subspace $X$ of $V$.
\end{prop}
In other words,
\[R^\uparrow(H) ~ = ~ \bigcup_{X\in G_{k-2}(V)}R_X(H)\]
where $G_{k-2}(V)$ stands for the set of the $(k-2)$-subspaces of $V$, namely  the points of $\cG_{k-2}(V)$. Recalling that the radical of a polar space of symplectic type always has even codimension in the underlying projective space, two cases must be distinguished.
\begin{enumerate}[1)]
\item
  {\em $n-k$ is even}. In this case $R_X(H)$ can be empty for some $X\in G_{k-2}(V)$. If this happens for all $X\in G_{k-2}(V)$, then $R^\uparrow(H) = \emptyset$.
\item
 {\em $n-k$ is odd}. Then $R_H(X)$ contains at least one point (a $(k-1)$-space) for every $X\in G_{k-2}(V)$. In this case $R_X(H)$ is far from being empty. However, it might happen that $R_X(H)$ is a singleton for every $X\in G_{k-2}(V)$. If this is the case then we say that $R^\uparrow(H)$ is
 \emph{spread-like}. This terminology is motivated by the fact that for $k = 3$,
 the set $R^\uparrow(H)$ is spread-like if and only if it is actually a spread of $\PG(V)$,  once it is regarded as a set of lines of $\PG(V)$.
\end{enumerate}

We are now ready to state the problems which we shall study in this paper.
They are essentially the same as those considered by Draisma and Shaw
\cite{DS10}, \cite{DS14}.

\begin{problem}\label{Prob1}
Let $k \geq 3$ and $H$ a hyperplane of $\cG_k(V)$. \\
{\rm (1.1)} Let $n-k$ be even. Is it possible that $R^\uparrow(H) = \emptyset$ for a suitable choice of $H$?  \\
{\rm (1.2)} Let $n-k$ be odd. Can it happen that $R^\uparrow(H)$ is spread-like?
\end{problem}

Problems (1.1) and (1.2) are somehow mutually related. Indeed, let $k < n-1$ be such that $n-k$ is even and let $H$ be a hyperplane of $\cG_k(V)$. Given a hyperplane $W$ of $V$, let $G_k(W)$ be the set of $k$-subspaces of $W$. Then $H(W) := H\cap G_k(W)$ is a hyperplane of the $k$-Grassmannian $\cG_k(W)$ of $W$. (Note that $G_k(W)\not\subseteq H$, otherwise $W\subseteq R_\downarrow(H)$, contradicting the fact that $R_\downarrow(H)$ has codimension at least $k$ in $V$). In Section~\ref{radicals} we will prove the following.

\begin{prop}\label{mutual relation}
With $n$ and $k$ as above, we have $R^\uparrow(H) = \emptyset$ if and only if $R^\uparrow(H(W))$ is spread-like for every hyperplane $W$ of $V$.
\end{prop}

We have stated problems (1.1) and (1.2) in the most general form, for any $k \geq 3$. However, it is clear from the definition of $R^\uparrow(H)$ that, if we have proved that $R^\uparrow(H)\neq \emptyset$ for $k = 3$ and $n$ odd for any hyperplane $H$ whenever $\KK$ belongs to a certain class $\bf C$ of fields, then the same holds for any choice of $k$ and $n$ with $k\geq 3$ and $n-k$ even, provided that $\KK\in {\bf C}$. Similarly, if we know that for $\KK\in {\bf C}$, $k = 3$ and $n$ even, the set $R^\uparrow(H)$ is never a spread of $\PG(V)$ then, as long as $\KK$ is chosen in $\bf C$, the radical $R^\uparrow(H)$ is never spread-like, for any $k\geq 3$ with $n-k$ odd. So, we can replace problems (1.1) and (1.2) with their following special cases.

\begin{problem}\label{Prob2}
Let $k = 3$ and $H$ a hyperplane of $\cG_3(V)$. \\
{\rm (2.1)} Let $n$ be odd. Can it happen that $R^\uparrow(H) = \emptyset$?  \\
{\rm (2.2)} Let $n$ be even. Is it possible that $R^\uparrow(H)$ is a line-spread of $\PG(V)$?
\end{problem}

In the next subsection we shall briefly survey the answers known from the literature to problems  (2.1) and (2.2) and state a few new results of our own, to be proved later in this paper.

\subsection{Answers}\label{answers}

We keep the notation of the previous subsection. In particular $n = \dim(V)$ and $\KK$ is the underlying field of $V$. However, throughout this section we assume $k = 3$ (whence $n \geq 5$).

In some of the results to be stated in this subsection the field $\KK$ will be
assumed to have cohomological dimension at most $1$. We recall that, for a prime $p$, the \emph{$p$-{cohomological dimension}} of a field $\KK$ is the $p$-cohomological dimension of the Galois group of the separable closure of $\KK$. The \emph{cohomological dimension} of $\KK$ is the supremum of its $p$-cohomological dimensions for $p$ ranging in the set of all primes (Serre \cite{Serre}, Gille and Szamuely \cite{C1}). For instance, algebraically closed fields have cohomological dimension $0$ and finite fields have cohomological dimension $1$, while the field $\RR$ of real numbers has infinite cohomological dimension.

Henceforth we denote the class of fields of cohomological dimension $0$ or $1$ by the symbol ${\bf Cd}\{0,1\}$. In some of the following theorems we assume that $\KK\in {\bf Cd}\{0,1\}$ is perfect. We warn that non-perfect fields exist that belong to ${\bf Cd}\{0,1\}$ (see Serre \cite[II, \S 3.1]{Serre}, also the example mentioned at the end of the next paragraph).

We shall also consider quasi-algebraically closed fields. Recall that a field $\KK$ is said to be \emph{quasi-algebraically closed} if every homogeneous equation with coefficients in $\KK$, in $t$ unknowns and of degree $d<t$ always  admits non-trivial solutions in $\KK^t$. For instance, all finite fields are quasi-algebraically closed (see \cite[II, \S 3.3]{Serre}). Quasi-algebraically closed fields form a proper subclass of ${\bf Cd}\{0,1\}$ (see \cite[II, 3.2]{Serre}). Note also that non-perfect quasi-algebraically closed fields exist. For instance, a transcendental extension of degree 1 of an algebraically closed field is quasi-algebraically closed \cite[II, \S 3.3]{Serre}, but in positive characteristic it is non-perfect.

We now turn to problems~\ref{Prob2},  (2.1) and (2.2). It is not difficult to prove that when $n = 5$, up to isomorphism, only two hyperplanes exist in $\cG_3(V)$ (see Section~\ref{low dim}, Theorem~\ref{l5}). We have $R^\uparrow(H) \neq \emptyset$ in both of these cases. Turning to $n = 6$, Revoy \cite{Revoy79} gives a complete classification of alternating $3$-linear forms.  We shall report on it in \S~\ref{n = 6}. It turns out that a form giving rise to a hyperplane $H$ of $\cG_3(V)$ where $R^\uparrow(H)$ is a spread of $\PG(V)$ exists precisely when $\KK$ is not quadratically closed  (Section~\ref{low dim}, Theorem~\ref{case n=6}). To sum up,

\begin{theorem}\label{Class56}
Let $n = 5$. Then, $R^\uparrow(H) \neq \emptyset$ for any hyperplane $H$ of $\cG_3(V)$ and any choice of the underlying field $\KK$ of $V$. \\
Let $n = 6$. Then, a hyperplane $H$ of $\cG_3(V)$ such that $R^\uparrow(H)$ is a spread exists if and only if $\KK$ is not quadratically closed.
\end{theorem}
\begin{note}
When $n = 6$ and $\KK$ is finite, examples where $R^\uparrow(H)$ is a spread are also constructed by Draisma and Shaw \cite{DS14}.
\end{note}

Assuming that $\KK$ is perfect and belongs to ${\bf Cd}\{0,1\}$, Cohen and Helminck \cite{CH88} give a complete classification of alternating $3$-linear forms for $n=7$. Referring the reader to \S~\ref{n = 7} for details, we only mention here the following byproduct of that classification.

\begin{theorem}\label{Class7}
Let $n = 7$ and let  $\KK$ be a perfect field in the class ${\bf Cd}\{0,1\}.$ Then $R^\uparrow(H) \neq \emptyset$ for every hyperplane $H$ of $\cG_3(V)$.
\end{theorem}
\begin{note}\label{Class7 note}
The hypothesis $\KK \in{\bf Cd}\{0,1\}$ cannot in general be removed from Theorem~\ref{Class7}. For instance, as Draisma and Shaw show in~\cite{DS10},~\cite{DS14}, when $n = 7$ and $\KK = \RR$ (but any subfield of $\RR$ would do the job as well) the Grassmannian $\cG_3(V)$ admits a hyperplane $H$ with $R^\uparrow(H) = \emptyset$. This is related to the exceptional cross product $\times:V(7,\RR)\times V(7,\RR)\rightarrow V(7,\RR)$ (see Brown and Gray \cite{VCP} or Lounesto \cite{Lou} for the definition and properties of this product).
However, there are also fields not in the class ${\bf Cd}\{0,1\}$ for
which $R^{\uparrow}(H)\neq\emptyset$ for any choice of the hyperplane $H$.
Indeed, it can be seen as a direct consequence of the second claim of
Theorem~\ref{Class56}  and Proposition~\ref{mutual relation} that if
$\KK$ is quadratically closed and $n=7$, then $R^{\uparrow}(H)\neq\emptyset$.
\end{note}

In Section~\ref{Sect4}
 of this paper, exploiting the classification of Cohen and Helminck \cite{CH88} we shall prove the following:
\begin{theorem}\label{main 1}
Let $n = 8$ and take $\KK$ to be a perfect field in the class ${\bf Cd}\{0,1\}$. Assume moreover that
\begin{itemize}
\item[$(*)$] every homogeneous equation of degree $3$ in $8$ unknowns with coefficients in $\KK$ admits non-trivial solutions in $\KK^8$.
\end{itemize}
Then, $R^\uparrow(H)$ is never a spread, for any hyperplane $H$ of $\cG_3(V)$.
\end{theorem}

Hypothesis $(*)$ of Theorem~\ref{main 1} holds if either $\KK$ is quasi-algebraically closed or every polynomial $p(t)\in \KK[t]$ of degree $3$ admits at least one zero in $\KK$. Therefore, if $n = 8$ and $\KK$ is algebraically closed then $R^\uparrow(H)$ is never a spread. This answers a question raised by Draisma and Shaw in \cite[Remark 9]{DS14}.

The next statement (to be proved in Section~\ref{Sect4}) is all we can say at the moment about $R^\uparrow(H)$ when $n > 9$ is even.

\begin{theorem}\label{main 2}
Let $\KK$ be a finite field and suppose $n\equiv 4\pmod6$. Then $R^\uparrow(H)$ is never a spread, for any hyperplane $H$ of $\cG_3(V)$.
\end{theorem}

By combining Theorem~\ref{main 1} with Proposition~\ref{mutual relation} we immediately obtain the following.

\begin{corollary}\label{cor m1}
With $\KK$ as in Theorem~\ref{main 1}, let $n = 9$. Then $R^\uparrow(H) \neq \emptyset$ for every hyperplane $H$ of $\cG_3(V)$.
\end{corollary}

By Theorem~\ref{main 2} and Proposition~\ref{mutual relation} we  obtain that, if $\KK$ is a finite field and $n\equiv 4\pmod6$ then
$R^\uparrow(H)\neq \emptyset$ for every hyperplane $H$ of $\cG_3(V)$. However this conclusion as well as the conclusion of Corollary~\ref{cor m1} when $\KK$ is quasi-algebraically closed, are contained in the following theorem of Draisma and Shaw \cite{DS10}:

\begin{theorem}\label{quasi alg. closed}
Let $n$ be odd and assume that $\KK$ is quasi algebraically closed. Then $R^\uparrow(H) \neq \emptyset$ for every hyperplane $H$ of $\cG_3(V)$.
\end{theorem}

\begin{note}
 In their proof of Theorem~\ref{quasi alg. closed} Draisma and Shaw consider the image $\RR^\uparrow(H) := \varepsilon_2(R^\uparrow(H))$ of $R^\uparrow(H)$ via the Pl\"{u}cker embedding $\varepsilon_2$. The crucial step in their proof is to show that $\RR^\uparrow(H)$ is an algebraic variety of degree $(n-1)/2-1$ but different proofs can be given in special cases. For instance, when $\KK$ is algebraically closed the conclusion $R^\uparrow(H)\neq\emptyset$ follows from a celebrated results on linear subspaces disjoint from projective varieties (see Subsection~\ref{algebraic approach}, Note~\ref{alg. closed proof}).
\end{note}

\begin{note}
Relying on Gurevitch's classification of trivectors of an $8$-dimensional complex vector space \cite[\S 35]{Gu64}, Djokovi\'{c} \cite{Djok} has classified trivectors of an $8$-dimensional real vector space. Following Djokovi\'{c}'s classification, if $n = 8$ and $\KK = \RR$, then $R^\uparrow(H)$ cannot be a spread for any hyperplane $H$ of $\cG_3(V)$. Consequently,
in view of Proposition~\ref{mutual relation}, if $n = 9$ and $\KK = \RR$ then $R^\uparrow(H) \neq \emptyset$ for every hyperplane $H$ of $\cG_3(V)$. However, this conclusion is contained in a stronger result of Draisma and Shaw \cite[Theorem 2]{DS14}, where it is proved that, in contrast with the exceptional behavior of $\RR$ when $n = 7$ (see Note~\ref{Class7 note}), if $n$ is odd, $n\geq 9$ and $\KK = \RR$ then $R^\uparrow(H) \neq \emptyset$ for every hyperplane $H$ of $\cG_3(V)$.
\end{note}

\subsection{More definitions}

We shall now state a few more definitions, to be used later in this paper. Still assuming $k = 3$, let $H$ be a hyperplane of $\cG_3(V)$.

\subsubsection{The geometry of poles}\label{poles}

Recall that the rank $\rk(X)$ of a (possibly empty) projective space $X$ is the projective dimension of $X$ augmented by $1$. In particular, $\rk(\emptyset) = 0$. Given a point $p$ of $\PG(V)$, let $r(p)$ be the rank of the radical $R_p(H)$ of the polar space ${\cS}_p(H) = ((p)G_2, (p)H)$ (notation as in Section~\ref{problems}). We call $r(p)$ the \emph{degree} of $p$ (relative to $H$). If $r(p) = 0$ then we say that $p$ is \emph{smooth}, otherwise we call $p$ a \emph{pole} of $H$, also $H$-\emph{pole} for short. Clearly, a point is a pole if and only if it belongs to a line $\ell\in R^\uparrow(H)$ (compare Proposition~\ref{polar space bis}). So, $R^\uparrow(H) = \emptyset$ if and only if all points are smooth.

As the polar space ${\cS}_p(H)$ is symplectic, $r(p)$ is even if $n$ is odd and it is odd if $n$ is even. In particular, when $n$ is even  all the points are poles. In both cases, the poles of degree $r(p) = n-1$ are just the points of $R_\downarrow(H)$.

Let $P(H)$ be the set of $H$-poles. Then $P(H)$ is the union of the lines of $\PG(V)$ that belong to $R^\uparrow(H)$. We can form a subgeometry ${\cP}(H) := (P(H), R^\uparrow(H))$ of $\PG(V)$, by taking $P(H)$ as the set of points and $R^\uparrow(H)$ as the set of lines, a point $p\in P(H)$ and a line $\ell\in R^\uparrow(H)$ being incident in ${\cP}(H)$ precisely when $p\in \ell$ in $\PG(V)$. We call ${\cP}(H)$ the \emph{geometry of poles} of $H$. We shall often refer to it in Section~\ref{Sect3}.

\begin{note}
We have assumed $k = 3$ but all the above can be easily rephrased for any $k \geq 3$, modulo a few obvious changes: we should consider $(k-2)$-subspaces of $V$ instead of points of $\PG(V)$ when defining poles and take lines $\ell_{X,Y}$ of $\cG_{k-2}(H)$ with $Y\in R^\uparrow(H)$ as lines of the geometry of poles. However, we shall not insist on this generalization here.
\end{note}

\subsubsection{The depth of $H$ and the rank of $R^\uparrow(H)$}\label{depth}

We define the \emph{depth} $\delta(H)$ of $H$ as the maximum degree $r(p)$, for $p$ ranging in the set of points of $\PG(V)$. Clearly, $\delta(H) = 0$ if and only if $R^\uparrow(H) = \emptyset$ and $\delta(H) = n-1$ if and only if $R_\downarrow(H) \neq \emptyset$.

We have noted in Section~\ref{problems} that $R^\uparrow(H)$ is a subspace of $\cG_2(H)$. So, we can also regard it as an induced subgeometry of $\cG_2(V)$, the lines of $R^\uparrow(H)$ being the lines of $\cG_2(V)$ contained in it. In this way, we can also give $R^\uparrow(H)$ a rank, as follows. Recall that a subspace of a point-line geometry is called \emph{singular} if all of its points are mutually collinear (Shult \cite{shult11}). A point-line geometry is said to be \emph{paraprojective} if all of its singular subspaces are projective spaces (Shult \cite[chapter 12]{shult11}). The \emph{rank} of a paraprojective geometry is the maximal rank of its singular subspaces.

Grassmannians are paraprojective and subspaces of paraprojective geometries are still paraprojective. Hence $R^\uparrow(H)$ is paraprojective. Accordingly, it admits a rank, henceforth denoted by the symbol $\rk(R^\uparrow(H))$. For the sake of completeness, when $R^\uparrow(H) = \emptyset$ we put $\rk(R^\uparrow(H)) = 0$. It is not difficult to see that $\delta(H) = \rk(R^\uparrow(H))$, except possibly in the case $\delta(H) = 1$ and $\rk(R^\uparrow(H)) = 2$.

\section{Preliminary results}\label{basics}

\subsection{Notation}\label{Notation}

For the convenience of the reader, we recall some notation introduced in the previous section. Let $X$ be a subspace of $V$ of dimension $\dim(X) = i \neq k$. When $i > k$ (possibly $i =n$) we denote by $G_k(X)$  and $\cG_k(X)$ the set of $k$-spaces contained in $X$ and the $k$-Grassmannian of $X$, respectively. Given a hyperplane $H$ of $\cG_k(V)$, we set $H(X) := G_k(X)\cap H$.

Suppose $i < k$. Then $(X)G_k$ is the set of $k$-spaces containing $X$ and $(X)\cG_k$ ($\cong \cG_{k-i}(V/X)$) is the subgeometry induced by $\cG_k(V)$ on $(X)G_k$. Also,  $(X)H := H \cap (X)G_k$. When $i = k-2$ (as defined in Section~\ref{problems}) the symbol ${\cS}_X(H)$ stands for the polar space with $(X)G_{k-1}$ as the set of points and $(X)H$ as the set of lines. $R_X(H)$ is the radical of ${\cS}_X(H)$.

We shall use square brackets in order to distinguish between a subspace of a vector space and the corresponding  projective subspace  in the projective geometry of that vector space. Thus, if $X$ is a subspace of some vector space then $[X]:=\{\langle x\rangle\colon x\in \langle X\rangle\setminus \{0\}\}$. In particular, if $v$ is a non-zero vector then $[v] = \langle v\rangle$ is the projective point represented by $v$. As usual, we write $[v_1,\ldots, v_k]$ instead of $[\langle v_1,\ldots, v_k\rangle]$, for short.

\subsection{Properties of radicals}\label{radicals}

In Section~\ref{problems} we have stated several results on radicals, referring the reader to the present section for the proofs of the less obvious among them. We shall now give those proofs.

Let $H$ be a hyperplane of $\cG_k(V)$ and $i < k$. The $i$-radical $R_i(H)$ of $H$, as defined in Section~\ref{problems}, is the set of $i$-subspaces $X\subset V$ such that $(X)H = (X)G_k$. The following is one of the claims made in Section~\ref{problems}.

\begin{prop}\label{V2(H) subspace}
The set $R_i(H)$ is a subspace of $\cG_i(V).$
\end{prop}
\begin{proof}
Let $X_1$ and $X_2$ be distinct elements of $R_i(H)$ belonging to the same line $\ell_{Y,Z}$ of $\cG_i(H)$. For $X\in l_{Y,Z}$, let $U\in (X)G_k$.
We need only to prove that $U\in H$.  If $U\supseteq Z$ then $U\in (X_1)G_k\cap (X_2)G_k$. Hence $U\in H$, as $X_1, X_2\in R_i(H)$.

Suppose that $U\not\supseteq Z$. Then $U\cap Z = X$. Let $W$ be a complement of $X$ in $U$. Then $W\cap Z = 0$. Hence $N := W+Y$ and $M := W+Z$ have dimension $k-1$ and $k+1$ respectively. As clearly $N \subset M$, we can consider the line $\ell_{N,M}$ of $\cG_k(V)$. The spaces $U_1 := X_1+W$ and $U_2 := X_2+W$ are points of $\ell_{N,M}$. As $X_1, X_2\in R_i(H)$, both $U_1$ and $U_2$ belong to $H$. We shall now prove that $U_1 \neq U_2$. By way of contradiction, let $U_1 = U_2 =: U'$, say. As both $U'$ and $Z$ contain $X_1$ and $X_2$, we have $Z\subset U'$. However $U'$ also contains $W$, by construction. As $\dim(U') = k$, $\dim(W) = k-i$ and $\dim(Z) = i+1$, we obtain that $Z\cap W \neq 0$, while we have previously proved that $Z\cap W = 0$. This contradiction forces $U_1\neq U_2$. Thus $\ell_{N,M}$ contains two distinct members of $H$, namely $U_1$ and $U_2$. Hence $\ell_{N,M}\subseteq H$, as $H$ is a hyperplane. In particular, $U \in H$.
\end{proof}

The following has also been mentioned in Section~\ref{problems}.
\begin{prop}\label{codimension lower radical}
The lower radical $R_\downarrow(H)$ of $H$ has codimension at least $k$ in $\PG(V)$.
\end{prop}
\begin{proof}
If the codimension of $R_{\downarrow}(H)$ is smaller than $k$, then every $k$-subspace of $V$ meets $R_\downarrow(H)$ non-trivially, thus forcing $H$ to be the full point-set of $\cG_k(V)$, while $H$ is, by definition, a proper subspace of $\cG_k(V)$.
\end{proof}

Proposition~\ref{mutual relation} of Section~\ref{problems} remains to be proved. We recall its statement here, for the convenience of the reader.
\begin{prop}\label{mutual relation bis}
Let $H$ be a hyperplane of $\cG_k(V)$ with $3 \leq k < n-1$. Assume that $n-k$ is even. Then $R^\uparrow(H) = \emptyset$ if and only if $R^\uparrow(H(W))$ is spread-like for every hyperplane $W$ of $V$.
\end{prop}
\begin{proof} Let $X\in G_{k-2}(W)$ be such that $R_X(H(W))$ contains a line $\ell$ of ${\cS}_X(H(W))$. (Actually, if it contains a line then it must contain also a plane, but we
shall not make use of this fact here). The polar space ${\cS}_X(H(W))$ is the subgeometry of ${\cS}_X(H)$ induced on the hyperplane $W/X$ of $V/X$. Therefore the line $\ell \subset {\cS}_X(H(W))$ contains at least one point collinear in ${\cS}_X(H)$ with all points of ${\cS}_X(H)$, namely a point of $R_X(H)$. Hence $R_X(H) \neq \emptyset$. By Proposition~\ref{polar space bis}, $R^\uparrow(H) \neq \emptyset$. The `only if' part of the statement is proved.

Turning to the `if' part, assume $R^\uparrow(H) \neq \emptyset$. Then $R_X(H) \neq \emptyset$ for some $X\in G_{k-2}(V)$. However, the projective space $R_X(H)$ has even rank, since $n-k$ is even. Therefore $R_X(H)$ contains a line $\ell$ of $\PG(V/X)$. Let $W$ be a hyperplane of $V$ containing $X$ and such that $[W/X] \supseteq \ell$. Clearly, $R_X(H(W)) \supseteq \ell$. Hence $R^\uparrow(H(W))$ is not spread-like.
\end{proof}

\subsection{An algebraic description of the upper radical $R^\uparrow(H)$}\label{algebraic approach}

As above, let $3\leq k < n = \dim(V)$. Given a linear functional $f:\bigwedge^kV\to\KK$, let $\tilde{f}$ be the linear mapping from $\bigwedge^{k-1} V$ to the dual $V^*$ of $V$ defined as follows:
\begin{equation}\label{tildef}
\begin{array}{rcl}
\tilde{f} ~~ : ~  v_1\wedge\cdots\wedge v_{k-1} \in \wedge^{k-1}V & \mapsto &  \tilde{f}(v_1\wedge\cdots\wedge v_{k-1}) \in V^* \\
\tilde{f}(v_1\wedge\cdots\wedge v_{k-1}) ~~ : ~  x\in V & \mapsto & f(v_1\wedge\cdots\wedge v_{k-1}\wedge x).
\end{array}
\end{equation}
Clearly,
\begin{equation}\label{sigma}
 \dim(\ker(\tilde{f}))~\geq~\dim \bigwedge^{k-1}V-\dim V^*={n\choose {k-1}}-n.
\end{equation}
Note that above $f$ is not assumed to be non-null. Clearly, if $f$ is the null functional then $\tilde{f}$ is null as well.

As recalled in Section~\ref{introduction}, for every hyperplane $H$ of $\cG_k(V)$ there exists a non-null linear functional
$f_H:\bigwedge^kV\to\KK$ such that $H=\varepsilon_{k}^{-1}(\ker(f_H)\cap\GG_k(V))$. With $\tilde{f}_H:\bigwedge^{k-1}V\rightarrow V^*$ defined as in \eqref{tildef}, put $K(H) :=\ker(\tilde{f}_H).$ We call $K(H)$ the \emph{kernel} of $H$.

The kernel $K(H)$ is strictly related to the upper radical $R^\uparrow(H)$ of $H$, as we shall presently see. Indeed,
a point $[v_1\wedge v_2\wedge\cdots \wedge v_{k-1}] \in \GG_{k-1}(V)$ belongs to $[K(H)]$ if and only if $\tilde{f}_H(v_1\wedge v_2\wedge\cdots\wedge v_{k-1})$ is the null linear functional. This is equivalent to say that $\langle v_1,\ldots, v_{k-1},x\rangle \in H$ for any $x\in V\setminus \langle v_1,\ldots, v_{k-1}\rangle$. Thus,
\begin{equation}\label{aG}
[K(H)]\cap \GG_{k-1}(V) ~=~ \RR^\uparrow(H) ~ (= \varepsilon_k(R^\uparrow(H))).
\end{equation}
This equation also proves that $\RR^\uparrow(H)$ is always an algebraic variety.

Take $X\in G_{k-2}(V)$ and let $\xi$ be a vector representative of the point $\varepsilon_{k-2}(X)\in \PG(\bigwedge^{k-2} V)$.
 Up to nonzero scalar multiples, the following alternating bilinear form $f_{H,X}$ defines $\cS_X(H)$ as a polar space embedded in $\PG(V/X)$:
\begin{equation}
\label{eqB}
 f_{H,X}:\begin{cases}
  V/X\times V/X\to\KK, \\
  (X+u,X+v)\to f_H(\xi\wedge u\wedge v).
  \end{cases}
\end{equation}
Henceforth, when writing $f_{H,X}(\langle X,u\rangle, \langle X, v\rangle)=0$ we shall mean that $f_{H,X}(X+u,X+ v)=0$.

Given a subspace $W$ of $V$ of dimension $\dim(W) > k$, let $f_{H|W}$ be the restriction of $f_H$ to $\bigwedge^kW$. Recall that $H(W) = H\cap G_k(W)$ is either a hyperplane of $\cG_k(W)$ or the whole set $G_k(W)$. In the first case the hyperplane $H(W)$ is defined by $f_{H|W}$ and, with $\tilde{f}_{H|W}$ defined as in \eqref{tildef}, the kernel of $\tilde{f}_{H|W}$ is the kernel $K(H(W))$ of $H(W)$. In the second case,  $H(W) = G_k(W)$ and
$f_{H|W}$, as well as $\tilde{f}_{H|W}$, are null; consequently, $K(H(W)) = \bigwedge^{k-1}W$.

In any case, $K(H)\cap \bigwedge^{k-1}W \subseteq K(H(W))$. Note that in general this inclusion is proper.

\begin{note}\label{alg. closed proof}
When $\KK$ is algebraically closed, every subspace of $\PG(\bigwedge^{k-1} V)$ of codimension at most $(k-1)(n-k+1)$ meets $\GG_{k-1}(V)$
non-trivially, see \cite[Proposition 11.4]{harris}. As $\cod([K(H)]) \leq n$ by \eqref{sigma}, we have $[K(H)]\cap\GG_{k-1}\neq\emptyset$. Hence $R^{\uparrow}(H)\neq\emptyset$, as claimed in Theorem~\ref{quasi alg. closed}.

We warn that the above argument does not work for arbitrary fields. For instance, if $\KK$ is finite,
 then $\PG(\bigwedge^{k-1}V)$ admits subspaces of codimension $n$ disjoint from $\GG_{k-1}(V)$, see~\cite{co}.
\end{note}

\subsection{A result for $k = 3$}\label{subsec 3}

This subsection is devoted to the proof of a statement (see below, Proposition~\ref{cxx}) which will be exploited several times in Section~\ref{Sect3}.

Assuming $k = 3$, let $H$ be a hyperplane of $\cG_3(V)$. As in the previous subsection, $f_H \in (\bigwedge^3V)^*$ defines $H$, the linear mapping $\tilde{f}_H:\bigwedge^2V\rightarrow V^*$ is given as in \eqref{tildef} and $K(H) = \ker(\tilde{f}_H)$ is the kernel of $H$. For every subspace $W$ of $V$ of dimension $\dim(W) > 3$, if $H(W) \neq G_3(W)$ then $K(H(W))$ is the kernel of $H(W)$, otherwise $K(H(W)) = \bigwedge^2W$.  According to \eqref{eqB},
 given a point $\langle v \rangle\in\cG_1(V)$, we can regard $\cS_{\langle v\rangle}(H)$ as the polar space associated to the following
 bilinear alternating form $f_{H,\langle v\rangle}:V/\langle v\rangle\times V/\langle v\rangle\rightarrow \KK$:
\begin{equation} \label{eD}
f_{H,\langle v\rangle}(\langle v\rangle+x,\langle v\rangle+y) ~ = ~ f_H(v\wedge x\wedge y) ~ (= ~\tilde{f}_H(x\wedge y)(v)).
\end{equation}
As in Subsection~\ref{depth}, we denote by $\delta(H)$ the depth of $H$. Recall that $n = \dim(V)$.

\begin{lemma}\label{cc0}
Let $t$ be a non-negative integer, $t \leq (n-2)/2$. Then either $\delta(H) \geq n-1-2t$ or $\dim(\bigwedge^2W/K(H(W))) = n-t$ for every subspace $W$ of $V$ with $\dim(W) = n-t$.
\end{lemma}
\begin{proof}
Suppose that for some $W<V$ with $\dim(W) =n-t$ we have
$\dim(\bigwedge^2W/K(H(W)))<n-t$, namely the mapping
$\tilde{f}_{H|W}:\bigwedge^2W\rightarrow W^*$ is not surjective. Then the image $T := \mathrm{im}\tilde{f}_{H|W}$ of $\tilde{f}_{H|W}$ is contained in a hyperplane $U$ of $W^*$. In particular, there exists $w\in W\setminus\{0\}$ such that
for all $u^*\in U$, and thus for all $u^*\in T$, we have $u^*(w)=0$.
Hence $f_{H|W}(x\wedge y\wedge w) = \tilde{f}_{H|W}(x\wedge y)(w)=0$ for all $x, y \in W$, since $\tilde{f}_{H|W}(x\wedge y)\in T.$ Therefore
$f_H(x\wedge y \wedge w) = 0$ for all $x, y \in W$, since $f_{H|W}$ is just the restriction of $f_H$ to $\bigwedge^3W$. Equivalently (compare (\ref{eD})),
\begin{equation}\label{eDD}
f_{H,\langle w\rangle}(\langle w\rangle+x,\langle w\rangle+y) ~ = ~ 0, ~~~ \forall x, y \in W.
\end{equation}
Put $p := \langle w\rangle$. In view of (\ref{eDD}), the projective space $[W/p]$ is a totally isotropic subspace of the polar space ${\cS}_{p}(H)$ (i.e. a singular subspace in the sense of \cite{shult11}). As $\dim(W/p) = n-t-1$, the projective space $[W/p]$ has rank $\rk[W/p] = n-t-1$.

Recall that the degree $r(p)$ of $p$ is the rank of the radical $R_p(H)$ of ${\cS}_p(H)$. Since ${\cS}_p(H)$ is a polar space of symplectic type embedded in $V/p$ and $\dim(V/p) = n-1$, a maximal singular subspace of ${\cS}_p(H)$ has rank equal to $(n-1+ r(p))/2$. However, $[W/p]$ is totally isotropic and $\rk[W/p] = n-t-1$. Therefore $(n-1+r(p))/2 \geq n-t-1$. It follows that $r(p) \geq n-2t-1$. Consequently, $\delta(H) \geq n-2t-1$.
\end{proof}

\begin{prop}\label{cxx}
Let $t$ be a non-negative integer, $t \leq (n-2)/2$. Then either $\delta(H) \geq n-1-2t$ or $\dim (K(H(W))/(K(H)\cap K(H(W)))) ~ \leq ~ t$ for any $(n-t)$-dimensional subspace $W$ of $V$.
\end{prop}
\begin{proof}
Suppose $\delta(H) < n-1-2t$. By Lemma~\ref{cc0}, for any $W\leq V$ with $\dim W=n-t$ we have $\dim K(H(W)))={{n-t}\choose 2}-n+t$. By Grassmann's formula and~(\ref{sigma}),
\[ \begin{array}{ll}
     \dim(K(H)\cap\bigwedge^2W)\geq &
     \dim K(H)+\dim \bigwedge^2W-\dim\bigwedge^2V\geq \\[3mm]
 & \displaystyle {n\choose 2}-n+{{n-t}\choose 2}-
  {n\choose 2}={{n-t}\choose 2}-n.
  \end{array} \]
Since $K(H)\cap\bigwedge^2W\subseteq K(H(W))\subseteq \bigwedge^2W$, we have $K(H)\cap K(H(W)) = K(H)\cap \bigwedge^2W$. Hence
$\dim(K(H)\cap K(H(W))) = \dim(K(H)\cap\bigwedge^2W) \geq {{n-t}\choose 2}-n$.
  The result now follows from the equality $\dim K(H(W)) = {{n-t}\choose 2}-n+t$.
\end{proof}

\section{Constructions and classifications}\label{Sect3}

In the first part of this section (subsections~\ref{trivial} and~\ref{symplectic hyperplane section}) we shall describe two ways to construct hyperplanes of $\cG_k(V)$ that work for any choice of $n = \dim(V)$ and $k < n$. In the second part, we turn our attention to the cases $k = 3$ and $5 \leq n \leq 7$, giving a survey of what is presently known on hyperplanes of $\cG_3(V)$ for these values of $n$. In particular, when $n = 5$ only two types of hyperplanes exist. A complete classification is also available for $n = 6$ while for $n = 7$ a classification is known only under the assumption that the underlying field $\KK$ of $V$ is perfect and belongs to the class ${\bf Cd}\{0,1\}$ of fields of cohomological dimension at most $1$.

\subsection{Trivial extensions, trivial hyperplanes and lower radicals}\label{trivial}

Let $V = V_0\oplus V_1$ be a decomposition of $V$ as the direct sum of two non-trivial subspaces $V_0$ and $V_1$. Put $n_0 := \dim(V_0)$ and assume that $n_0 \geq k$ ($\geq 3$). Let $\varphi_0:V_0\times\cdots\times V_0\rightarrow \KK$ be a non-trivial $k$-linear alternating form on $V_0$. The form $\varphi_0$ can naturally be extended to a $k$-linear alternating form $\varphi$ of $V$ by setting
\begin{equation}\label{extended form}
\left.\begin{array}{rcll}
\varphi(x_1,\dots, x_k) & = & 0 & \mbox{if} ~ x_i\in V_1 ~ \mbox{for some} ~ 1\leq i\leq k,\\
\varphi(x_1,\dots, x_k) & = & \varphi_0(x_1,\dots, x_k) & \mbox{if} ~ x_i\in V_0 ~\mbox{for all} ~ 1\leq i\leq k,
\end{array}\right\}
\end{equation}
and then extending by linearity. Let $H_\varphi$ be the hyperplane of $\cG_k(V)$ defined by $\varphi$. The following is straightforward.

\begin{theorem}\label{extended hyperplane}
Assuming that $k < n_0$, let $H_0$ be the hyperplane of $\cG_k(V_0)$ defined by $\varphi_0$. Let also $\pi: V\rightarrow V_0$ be the projection of $V$ onto $V_0$ along $V_1$. Then,
\[H_\varphi ~ = ~ \{X\in G_k(V)~\colon~ \mbox{either}~ X\cap V_1 \neq 0 ~\mbox{or}~ \pi(X)\in H_0\}.\]
\end{theorem}
The properties gathered in the next corollary immediately follow from Theorem~\ref{extended hyperplane}.

\begin{corollary}\label{ext hyp cor}
Let $k < n_0$. Then all the following properties hold:
\begin{enumerate}[(1)]
\item\label{ext-1} $R_\downarrow(H_\varphi) ~ = ~ \langle R_\downarrow(H_0) \cup [V_1]\rangle$ ~where the span in taken in $\PG(V)$.
\item\label{ext-2} $R^\uparrow(H_\varphi) ~= ~ \{X\in G_{k-1}(V)~ \colon~  \mbox{either}~ X\cap V_1\neq 0 ~\mbox{or}~ \pi(X)\in R^\uparrow(H_0)\}.$
\item\label{ext-3} For $X\in G_{k-2}(V)$, if $X\cap V_1\neq 0$ then $R_X(H_\varphi) = [V/X]$ (namely ${\cS}_X(H_\varphi)$ is trivial), otherwise $R_X(H_\varphi)$ has the same rank as the span of $R_{\pi(X)}(H_0)\cup [(V_1+ \pi(X))/\pi(X)]$ in $\PG(V/\pi(X))$.
\end{enumerate}
\end{corollary}
\begin{note}
 When $k = 3$, claim~(\ref{ext-3}) of Corollary~\ref{ext hyp cor} can be rephrased as follows: the points $p\not\in [V_1]$ have degree $r(p) = r_0(\pi(p)) + n-n_0$, where $r_0(\pi(p))$ is the degree of $\pi(p)$ with respect to $H_0$. The points $p\in [V_1]$ have degree $n-1$.
\end{note}

We call $H_\varphi$ the \emph{trivial extension of $H_0$  centered at $V_1$}
 (also \emph{extension of $H_0$ by $V_1$}, for short) and we denote it by the symbol $H_0\odot V_1$. When convenient, we shall
 take the liberty of writing $H_0\odot[V_1]$ instead of $H_0\odot V_1$, referring to the projective space $[V_1]$ instead of $V_1$.

In Theorem~\ref{extended hyperplane} and Corollary~\ref{ext hyp cor} we have assumed $k < n_0$ in order to introduce the hyperplane $H_0$ associated to $\varphi_0$, but the case $k = n_0$ can also be dealt with, modulo some conventions. Let $k = n_0$; then $\varphi_0(x_1,\ldots, x_k) = 0$ if and only if the vectors $x_1,\ldots, x_k$ are linearly dependent. We can still give $H_0$ a meaning, stating that in this case $H_0 = \emptyset$. We can also stress the terminology stated in the introduction of this paper, putting $\cG_k(V_0) = \{V_0\}$ and regarding $\emptyset$ as the unique hyperplane of $\{V_0\}$. Accordingly, $R_\downarrow(H_0) = R^\uparrow(H_0) = \emptyset$. Also ${\cS}_X(H_0) = G_{k-1}(X)$ (just a set of points, with no lines) and  $R_X(H_0) = \emptyset$ for every $X\in G_{k-2}(V_0)$. With this conventions, Theorem~\ref{extended hyperplane} and Corollary~\ref{ext hyp cor} remain valid, word by word. Thus, we feel we are allowed to denote $H_\varphi$ by the symbol $H_0\odot V_1$ and call it the \emph{trivial extension} of $H_0$ by $V_1$ even in the case $k = n_0$. However, it will be convenient to have also a different name and a different symbol for this situation: when $k = n_0$ (namely $H_0 = \emptyset$) we shall
call $H_0\odot V_1$ the \emph{trivial hyperplane centered at} $V_1$ (or \emph{at} $[V_1]$, if we prefer so).

When $H_0\odot V_1$ is trivial in the above sense, Theorem~\ref{extended hyperplane} and Corollary~\ref{ext hyp cor} can also be rephrased as follows, with no mention of $H_0$.

\begin{prop}\label{ext hyp sing}
Let $H$ be the trivial hyperplane of $\cG_k(V)$ centered at $V_1$. Then
\[H ~ = ~ \{X\in G_k(V)~\colon~ X\cap V_1 \neq 0\}.\]
Moreover, $R_\downarrow(H) ~ = ~  [V_1]$, $R^\uparrow(H) ~= ~ \{X\in G_{k-1}(V)~ \colon~ X\cap V_1\neq 0\}$ and, for $X\in G_{k-2}(V)$, if $X\cap V_1\neq 0$ then $R_X(H) = [V/X]$, otherwise $R_X(H) = [(V_1+ X)/X]$.
\end{prop}

By the previous results, the lower radical of a trivial extension is never empty. The converse is also true: if $R_\downarrow(H)\neq\emptyset$ then $H$ is a trivial extension, possibly a trivial hyperplane.

Indeed, let $H$ be a hyperplane of $\cG_k(V)$ and let $\varphi$ be a $k$-linear alternating form on $V$ defining $H$. Let $R := \Rad(\varphi)$ be the radical of $\varphi$. Then $R_\downarrow(H) = [R]$, as noticed in Section~\ref{problems}. Let $S$ be a complement of $R$ in $V$ and $\varphi_S$ the form induced by $\varphi$ on $S$. If $\varphi_S$ is
trivial, then $\varphi$ will be trivial as well, as $R = \Rad(\varphi)$. However $\varphi$ is not trivial, since it defines a hyperplane. Hence, $\varphi_S$ is non-trivial. Put $n_S := \dim(S)$. By Proposition~\ref{codimension lower radical}, the subspace $R_\downarrow(H)$ has codimension at least $k$ in $\PG(V)$. Equivalently, $\cod(R) \geq k$. Hence, $n_S \geq k$. If $n_S > k$, then $H(S) = G_k(S)\cap H$ is the hyperplane of $\cG_k(S)$ associated to the non-trivial form $\varphi_S$. If $n_S = k$, then we put $H(S) := \emptyset$.

\begin{theorem}\label{ext hyp th}
Suppose $R_\downarrow(H) \neq \emptyset$ and let $S$ be a complement in $V$ of the subspace $R < V$ such that $[R] = R_\downarrow(H)$. Then $H ~=~ H(S)\odot R_\downarrow(H)$. Moreover, $R_\downarrow(H(S)) = \emptyset$.
\end{theorem}
\begin{proof}
By assumption, $R \neq 0$. The form $\varphi$ satisfies conditions (\ref{extended form}) with $\varphi_S$ and  $R$ in the roles of $\varphi_0$ and $V_1$ respectively. Hence $\varphi$ is the extension of $\varphi_S$ as defined by those conditions. Consequently, $H$ is the trivial extension of $H(S)$ by $R_\downarrow(H)$. When $n_S > k$ the equality $R_\downarrow(H(S)) = \emptyset$ follows from (\ref{ext-1}) of Corollary~\ref{ext hyp cor} and the fact that $H(S)\odot R_\downarrow(H) = H$. When $n_S = k$, then $R_\downarrow(H(S)) = \emptyset$ by definition.
\end{proof}

\begin{corollary}\label{ext hyp cor1}
With $R$ as in the hypotheses of Theorem~{\rm~\ref{ext hyp th}}, let $S$ and $S'$ be two complements of $R$ in $V$. Then, $H(S) \cong H(S')$.
\end{corollary}
\begin{proof}
$H(S)\odot R_\downarrow(H) = H(S')\odot R_\downarrow(H) = H$ by Theorem~\ref{ext hyp th}. The projection of $V$ onto $S'$ along $R$ induces on $S$ an isomorphism $\pi_S:S\stackrel{\cong}{\rightarrow} S'$. By Theorem~\ref{extended hyperplane}, $\pi_S$ maps $H(S)$ onto $H(S')$.
\end{proof}

\begin{corollary}\label{ext hyp cor2}
A hyperplane $H$ of $\cG_k(V)$ is trivial if and only if $\cod(R_\downarrow(H)) = k$.
\end{corollary}
\begin{proof}
This immediately follows from Theorem~\ref{ext hyp th}.
\end{proof}

\subsection{Expansions and symplectic hyperplanes}\label{symplectic hyperplane section}

Let $V_0$ be a hyperplane of $V$ and $H_0$ a given hyperplane of $\cG_{k-1}(V_0)$. As usual, assume $k \geq 3$; hence $V$ has dimension $n \geq 4$. Put:
\[E(H_0) ~:= ~ \{X\in G_k(V) ~\colon~ \mbox{either}~ X\subset V_0 ~\mbox{or}~X\cap V_0 \in H_0\}.\]

\begin{theorem}\label{sympl hyperplane}
The set $E(H_0)$ is a hyperplane of $\cG_k(V)$. Moreover:
\begin{enumerate}[(1)]
\item\label{sym-1}
 $R_\downarrow(E(H_0)) ~= ~ R_\downarrow(H_0)$.
\item\label{sym-2}
 $R^\uparrow(E(H_0)) ~= ~ H_0\cup\{X\in G_{k-1}(V)\setminus G_{k-1}(V_0) ~\colon~ X\cap V_0 \in R^\uparrow(H_0)\}$.
\item\label{sym-3}
 For $X\in \cG_{k-2}(V)$, if $X\subseteq V_0$ with $X\in R^\uparrow(H_0)$ then $R_X(E(H_0)) = {\cS}_X(E(H_0)) = (X)G_{k-1}$ (the latter being computed in $V$). If $X\subseteq V_0$ but $X\not\in R^\uparrow(H_0)$ then $R_X(E(H_0)) = (X)H_0$ (a subspace of $\PG(V_0/X)$). Finally, if $X\not\subseteq V_0$, then $R_X(E(H_0)) ~=~ \{\langle x, Y\rangle~\colon~ Y\in R_{X\cap V_0}(H_0)\}$ for a given $x\in X\setminus V_0$, no matter which.
\end{enumerate}
\end{theorem}
\begin{proof}
We only prove that $E(H_0)$ is a hyperplane of $\cG_k(V)$. The proofs of claims (\ref{sym-1}), (\ref{sym-2}) and (\ref{sym-3}) are straightforward. We leave them to the reader.

Let $\ell_{Y,Z}$ be a line of $\cG_k(V)$. If $Z\subseteq V_0$ then $\ell_{Y,Z}\subseteq E(H_0)$, by definition of $E(H_0)$. Let $Z\not\subseteq V_0$ but $Y\subseteq V_0$. Then $Z\cap V_0\supset Y$. In this case either $\ell_{Y,Z}\subseteq E(H_0)$ or $\ell_{Y,Z}\cap E(H_0) = \{Y\}$ according to whether $Y\in H_0$ or not. Finally, let $Y\not\subseteq V_0$. Then $\ell_{Y\cap V_0, Z\cap V_0}$ is a line of $\cG_{k-1}(V_0)$ and, given $y\in Y\setminus V_0$, we have $\ell_{Y,Z} ~= ~ \{\langle y, X\rangle ~\colon~ X\in \ell_{Y\cap V_0, Z\cap V_0}\}$. Consequently, either $\ell_{Y,Z}\subseteq E(H_0)$ or $\ell_{Y,Z}\cap E(H_0)$ is a singleton, according to whether $\ell_{Y\cap V_0,Z\cap V_0}\subseteq H_0$ or
 $\ell_{Y\cap V_0,Z\cap V_0}\cap H_0$ is a singleton. In any case, either $\ell_{Y,Z}$ is fully contained in $E(H_0)$ or it meets $E(H_0)$ in one single point. Thus, $E(H_0)$ is a hyperplane of $\cG_k(V)$.
\end{proof}

We call $E(H_0)$ the \emph{expansion} of $H_0$. A form $f:\bigwedge^kV\rightarrow \KK$ associated to $E(H_0)$ can be constructed as follows.
Suppose $f_0:\bigwedge^{k-1}V_0\rightarrow \KK$ defines $H_0$. Given a basis $\{e_1,\ldots, e_n\}$ of $V$ in such a way that $\langle e_1,\ldots, e_{n-1}\rangle = V_0$, consider the canonical basis $\{e_{i_1}\wedge\cdots\wedge e_{i_k} ~\colon ~ 1\leq i_1 < \ldots < i_k \leq n\}$ of $\bigwedge^kV$. Put
\begin{equation}\label{expanded form}
f(e_{i_1}\wedge\cdots\wedge e_{i_k}) = \left\{\begin{array}{l}
0 ~\mbox{if}~ i_k < n,\\
f_0(e_{i_1}\wedge\cdots\wedge e_{i_{k-1}}) ~\mbox{if}~ i_k = n.
\end{array}\right.
\end{equation}
and extend it by linearity. It is not difficult to check that $f$ indeed defines $E(H_0)$.

\subsubsection{Expansions and trivial extensions}

The next corollary immediately follows from (\ref{sym-1}) of Theorem~\ref{sympl hyperplane}.

\begin{corollary}\label{symp hyp cor1}
We have $R_\downarrow(E(H_0)) = \emptyset$ if and only if $R_\downarrow(H_0) = \emptyset$.
\end{corollary}
Hence, according to Theorem~\ref{ext hyp th}, if $R_\downarrow(H_0) \neq \emptyset$ then $E(H_0)$ is a trivial extension. Explicitly, let $R_0$ be the subspace of $V_0$ corresponding to $R_\downarrow(H_0)$. Assume that $R_0 \neq 0$ and let $S_0$ be a complement of $R_0$ in $V_0$. Let $S$ be a complement of $R_0$ in $V$ containing $S_0$. In order to avoid annoying complications, assume that $\dim(S_0) > k-1$. It is easy to check that the hyperplane induced by $E(H_0)$ on $S$ is just the expansion $E(H_0(S))$ (from $S_0$ to $S$) of the hyperplane $H_0(S_0)$ induced by $H_0$ on $S_0$. Hence Theorem~\ref{ext hyp th} and claim (\ref{sym-1}) of Theorem~\ref{sympl hyperplane} yield the following:

\begin{corollary}\label{symp hyp cor2}
Under the previous assumptions, $E(H_0) = E(H_0(S_0))\odot R_\downarrow(H_0)$.
\end{corollary}

In the above we have assumed that $\dim(S_0) > k-1$, namely $H_0$ is non-trivial. With some additional conventions, we can interpret the statement of Corollary~\ref{symp hyp cor2} so that it also holds when $H_0$ is trivial. However, there is no need to stress definitions this much. When $H_0$ is trivial we have the following simple statement.
\begin{corollary}\label{symp hyp cor3}
If $H_0$ is trivial then $E(H_0)$ is also trivial, with $R_\downarrow(H_0)$ as its center.
\end{corollary}
\begin{proof}
Let $H_0$ be trivial. Then $R_\downarrow(E(H_0))$ has codimension $k$ in $\PG(V)$, since $R_\downarrow(E(H_0)) = R_\downarrow(H_0)$ and the latter has codimension $k-1$ in $V_0$. The conclusion follows from Corollary~\ref{ext hyp cor2} and Proposition~\ref{ext hyp sing}.
\end{proof}

\subsubsection{Symplectic hyperplanes}\label{symp hyp sec}

Assume now $k = 3$. According to Proposition~\ref{Polar Space}, the point-line geometry ${\cS}(H_0) = (G_1(V_0), H_0)$ is a polar space of symplectic type. The upper and lower radical of $H_0$ are mutually equal and coincide with the radical $R(H_0)$ of ${\cS}(H_0)$.

Suppose firstly that $R(H_0) = \emptyset$, namely ${\cS}(H_0)$ is non-degenerate. Then $n-1$ is even, whence $n \geq 5$. Claims (\ref{sym-1}) and (\ref{sym-2}) of Theorem~\ref{sympl hyperplane} imply the following:

\begin{prop}\label{symp hyp prop}
$R_\downarrow(E(H_0)) = \emptyset$ and $R^\uparrow(E(H_0)) = H_0$.
\end{prop}

Let ${\cP}(E(H_0))$ be the geometry of poles of $E(H_0)$ (see Section~\ref{poles}). The next corollary immediately follows from Proposition~\ref{symp hyp prop}:

\begin{corollary}\label{symp hyp cor}
${\cP}(E(H_0)) = {\cS}(H_0)$.
\end{corollary}
In particular, the points of $[V]\setminus[V_0]$ are smooth while those of $[V_0]$ are poles (of degree $1$). Motivated by Corollary~\ref{symp hyp cor}, when $R(H_0) = \emptyset$ we call $E(H_0)$ a \emph{symplectic hyperplane}.

Assume now that $R(H_0) \neq \emptyset$. Then $R_\downarrow(E(H_0)) \neq \emptyset$, since $R_\downarrow(E(H_0)) = R(H_0)$ (by case~\ref{sym-1}) of Theorem~\ref{sympl hyperplane}). Corollaries~\ref{symp hyp cor2} and~\ref{symp hyp cor3} now imply that $E(H_0)$ is either a trivial extension
of a symplectic hyperplane by $R(H_0)$ (when $\rk(R(H_0)) < n-3$) or a trivial hyperplane centered at $R(H_0)$ (when $\rk(R(H_0)) =  n-3$).

\subsection{Low dimensional cases}\label{low dim}

In this subsection we survey what is known on hyperplanes of $\cG_3(V)$ when $V$ has dimension $n \leq 7$. When $n = 4$ the hyperplanes of $\cG_3(V)$ are just the hyperplanes of $\PG(V)$. The case $n = 5$ is dealt with in \S~\ref{n = 5} while \S\S~\ref{n = 6} and~\ref{n = 7} are respectively devoted to the cases $n = 6$ and $n = 7$.
 In \S~\ref{prel k-linear alt} we fix some notation for $k$-linear alternating forms, to be used in \S\S~\ref{n = 6} and~\ref{n = 7}.

In the sequel we shall refer to isomorphism classes of hyperplanes. Actually, we have already used the notion of isomorphism before (in Corollary~\ref{ext hyp cor1}, for instance) without referring to an explicit definition, but this  becomes now necessary. We say that two hyperplanes $H$ and $H'$ of $\cG_k(V)$ are \emph{isomorphic}, and we write $H \cong H'$, when $H' = g(H) := \{g(X)\}_{X\in H}$ for some $g\in\GL(V)$. Recall that two $k$-linear alternating forms $\varphi$ and $\varphi'$ on $V$ are said to be \emph{equivalent} when
\[\varphi'(x_1,\ldots, x_k) ~ = ~ \varphi(g(x_1),\ldots, g(x_k)), ~~~ \forall x_1,\ldots, x_k \in V\]
for some $g\in\GL(V)$. Accordingly, if $H$ and $H'$ are the hyperplanes associated to $\varphi$ and $\varphi'$, we have $H\cong H'$ if and only if $\varphi'$ is proportional to a form equivalent to $\varphi$. Note that if $\varphi' = \lambda\cdot \varphi$ for a scalar $\lambda \neq 0$ then $\varphi$ and $\varphi'$ are equivalent if and only if $\lambda$ is a $k$-th power in $\KK$.

We say that two forms $\varphi$ and $\varphi'$ are \emph{nearly equivalent}, and we write $\varphi\sim \varphi'$, when each of them is equivalent to a non-zero
scalar multiple of the other. With this notation, we can concisely rephrase the above definition as follows: $H\cong H'$ if and only if $\varphi\sim\varphi'$.

We extend the above terminology to linear functionals of $\bigwedge^kV$, saying that two linear functionals $f, f' \in (\bigwedge^kV)^*$ are \emph{nearly equivalent} and writing $f\sim f'$ when their corresponding $k$-alternating forms are nearly equivalent.

\subsubsection{Case $n = 5$}\label{n = 5}

\begin{theorem}\label{l5}
Let $n=5$. Then, up to isomorphism, only two hyperplanes exist in $\cG_3(V)$, namely the symplectic hyperplane and the trivial one.
\end{theorem}
\begin{proof}
This result can be drawn out of the classification of \cite{Revoy79}, but it can also be proved by the following elementary argument.

Trivial hyperplanes exist in any $k$-Grassmannian while symplectic hyperplanes exist in any $3$-Grassmannian $\cG_3(V)$ provided that $n = \dim(V)$ is odd and at least $5$.
As $n = 5$ by assumption, $\cG_3(V)$ admits both trivial and symplectic hyperplanes. In order to see that these two families form two isomorphism classes and no more hyperplanes exist in $\cG_3(V)$, we consider  $\cG_2(V^*)$. Any hyperplane $H$ of $\cG_3(V)$ appears also as a hyperplane $H^*$ of $\cG_2(V^*)$. According to Proposition~\ref{Polar Space}, $H^*$ is the line-set of a polar space ${\cS}(H^*)$ of symplectic type embedded in $\PG(V^*)$. As $\dim(V^*) = 5$, up to isomorphism, only two possibilities exist for ${\cS}(H^*)$, according as its radical has rank $1$ or $3$. Therefore, only two isomorphism classes of hyperplanes exist in $\cG_3(V)$.
\end{proof}

\begin{corollary}\label{l5 cor}
Let $n = 5$. Then $R^\uparrow(H) \neq \emptyset$ for every hyperplane $H$ of $\cG_3(V)$.
\end{corollary}
\begin{proof}
Indeed, $R^\uparrow(H)\neq \emptyset$ both when $H$ is trivial and when $H$ is of symplectic type.
\end{proof}

Of course, far more of what we have put in Corollary~\ref{l5 cor} can be said. As remarked in the proof of Theorem~\ref{l5}, the radical $R(H^*)$ of the polar space ${\cS}(H^*)$ has rank either $1$ or $3$. When $\rk(R(H^*)) = 3$ then $H$ is trivial while if $\rk(R(H^*)) = 1$ then $H$ is symplectic. In the first case $R_\downarrow(H)$ is a line of $\PG(V)$ and
$R^\uparrow(H)$ is the set of lines of $\PG(V)$ that meet the line $R_\downarrow(H)$ non-trivially. $R_\downarrow(H)$ corresponds to $R(H^*)$.

Suppose $\rk(R(H^*)) = 1$; then $H$ is of symplectic type. Accordingly, the geometry ${\cP}(H)$ of poles of $H$ is isomorphic to the generalized quadrangle $\cW(3,\KK)$ associated to a non-degenerate alternating bilinear form of $V(4,\KK)$. This can also be seen in a more direct way as follows. In the case under consideration $R(H^*)$, being a point of $\PG(V^*)$, corresponds to a hyperplane $V_0$ of $V$. The corresponding hyperplane $[V_0]$ of $\PG(V)$ is the set of poles of $H$, namely the point-set of the generalized quadrangle ${\cP}(H)$. The latter is naturally isomorphic to the quotient ${\cS}(H^*)/R(H^*)$ ($\cong \cW(3,\KK)$).

As ${\cP}(H) \cong \cW(3,\KK)$ and $\dim(V_0) = 4$, for every point $p\in [V_0]$ the lines of ${\cP}(H)$ through $p$ (namely the members of $R^\uparrow(H)$ containing $p$) span a plane $S_p$ of $[V_0]$. Hence they form a line $\ell_{p,S_p}$ of $\cG_3(V)$. These are precisely the lines of the geometry ${\cal R}^\uparrow(H)$ induced by $\cG_2(V)$ on $R^\uparrow(H)$. It is now clear that ${\cal R}^\uparrow(H)$ is isomorphic to the dual of ${\cP}(H)$. In short, ${\cal R}^\uparrow(H)  \cong {\cal Q}(4,\KK)$, the generalized quadrangle associated to a non-singular quadratic form of $V(5,\KK)$.

\subsubsection{Notation for linear functionals of $\bigwedge^3V$}\label{prel k-linear alt}

Given a non-trivial linear functional $h\in(\bigwedge^3 V)^*$, let $\chi$ and $H_h$ be respectively the alternating $3$-linear form and the hyperplane of ${\cG}_3(V)$ associated to it. Recall that $R_{\downarrow}(H_h) = [\Rad(\chi)]$ (see Section~\ref{introduction}). We put $\rk(h) := \cod_V(\Rad(\chi))$ and we call this number the \emph{rank} of $h$. Clearly, functionals of different rank can never be nearly equivalent. By Proposition~\ref{codimension lower radical} we know that $\rk(h) \geq 3$.

\begin{prop}\label{rank 4}
$\rk(h) \neq 4$.
\end{prop}
\begin{proof} By way of contradiction, suppose $\rk(h) = 4$. Then, by Theorem~\ref{ext hyp th},
  the hyperplane $H_h$ is a trivial extension $H = H_0\odot \Rad(\chi)$ for a complement $V_0$ of $\Rad(\chi)$ in $V$ and a suitable hyperplane $H_0 \in \cG_3(V_0)$ with $R_\downarrow(H_0) = \emptyset$. However, $\dim(V_0) = \rk(h) = 4$. Hence $H_0 = (p)G_3$ for a point $p\in \PG(V_0)$. Consequently, $R_\downarrow(H_0) = \{p\}$, while we said above that $R_\downarrow(H_0) = \emptyset$.
\end{proof}

Fix a basis $E:=(e_i)_{i=1}^n$ of $V$. The dual basis of $E$ in $V^*$ is  $E^* := (e^i)_{i=1}^n$, where
$e^i\in V^*$ is the linear functional such that $e^i(e_j) = \delta_{i,j}$ (Kronecker symbol).
The set $(e^{i}\wedge e^{j}\wedge e^{k})_{1\leq i<j<k\leq n}$ is the basis of $(\bigwedge^3V)^*$ dual of the basis $(e_i\wedge e_j\wedge e_k)_{1\leq i < j < k\leq n}$ of $\bigwedge^3V$ canonically associated to $E$. We shall adopt the convention of writing $\underline{ijk}$ for $e^i\wedge e^j\wedge e^k$, thus representing linear functionals of $\bigwedge^3V$ as linear combinations of symbols as $\underline{ijk}$.

In Table~\ref{Tab F} we list a number of possible types of linear functionals of $\bigwedge^3V$ of rank at most $7$, called $T_1,\ldots, T_9$ and $T^{(1)}_{10,\lambda}$, $T^{(2)}_{10,\lambda}$, $T^{(1)}_{11,\lambda}$ and $T^{(2)}_{11,\lambda}$, where $\lambda$ is a scalar subject to the conditions specified in the table.
Note that  description of each of these types makes sense for any $n$, provided that $n$ is not smaller than the rank of (a linear functional admitting) that description. Also, according to the clauses assumed on $\lambda$, types $T^{(r)}_{s,\lambda}$ ($r= \{1,2\}$, $s = \{10,11\}$) are considered only when $\KK$ is not quadratically closed. Moreover, when $\lambda \neq \lambda'$ the types $T^{(r)}_{s,\lambda}$ and $T^{(r)}_{s,\lambda'}$ are regarded as different in principle, even if it turned out that they describe nearly equivalent forms.

\begin{table}
\begin{small}
\[ \begin{array}[t]{c|l|c|l}
     \text{Type} & \text{Description} & \text{Rank} & \text{Special conditions, if any} \\ \hline
{}& & & \\
     T_1 & \underbar{123} & 3 & \\
     T_2 & \underbar{123}+\underbar{145} & 5 & \\
     T_3 & \underbar{123}+\underbar{456} & 6 & \\
     T_4 & \underbar{162}+\underbar{243}+\underbar{135} & 6 &  \\
     T_5 & \underbar{123}+\underbar{456}+\underbar{147} & 7 & \\
     T_6 & \underbar{152}+\underbar{174}+\underbar{163}+\underbar{243} & 7 & \\
    T_7 & \underbar{146}+\underbar{157}+\underbar{245}+\underbar{367} & 7 & \\
    T_8 & \underbar{123}+\underbar{145}+\underbar{167} & 7 & \\
     T_9 & \underbar{123}+\underbar{456}+\underbar{147}+\underbar{257}+\underbar{367} & 7 & \\
 {} & & & \\
     T_{10,\lambda}^{(1)} & \underbar{123}+\lambda(\underbar{156}+\underbar{345}+\underbar{426}) & 6 & p_\lambda(t):=t^2-\lambda \text{ irreducible in } \KK[t].\\
{} & & & \\
T_{10,\lambda}^{(2)} &\begin{array}[t]{@{}l}\underbar{126}+\underbar{153}+\underbar{234}+\\
(\lambda^2+1)\underbar{456}+\lambda(\underbar{156}+\underbar{345}+\underbar{426}) \end{array} & 6 & \begin{array}[t]{@{}l} \chr(\KK)=2 \text{ and}\\
                     p_{\lambda}(t):=t^2+\lambda t+1 \text{ irreducible in } \KK[t].\\
\end{array} \\
{}& & & \\
      T_{11,\lambda}^{(1)} & [\text{the same as at} ~ T^{(1)}_{10,\lambda}] +\underbar{147} & 7 & \text{same conditions as for} ~ T_{10,\lambda}^{(1)}\\
      T_{11,\lambda}^{(2)} & [\text{the same as at} ~ T^{(2)}_{10,\lambda}]+\underbar{147} & 7 & \text{same conditions as for} ~ T_{10,\lambda}^{(2)}\\
\end{array}\qquad
\]
\end{small}
\caption{types for linear functionals of $\bigwedge^3 V$.}
\label{Tab F}
\end{table}

If $T$ is one of the types of Table~\ref{Tab F}, we say that $h\in (\bigwedge^3V)^*$ is of \emph{type} $T$ if $h$ is nearly equivalent to the linear functional described at row $T$ of Table~\ref{Tab F}. The \emph{type} of $H_h$ is the type of $h$. Clearly, functionals of the same type are nearly equivalent. It follows from Revoy \cite{Revoy79} and Cohen and Helminck \cite{CH88} that two functionals of types $T_i$ and $T_j$ with $1\leq i < j \leq 9$ are never nearly equivalent; a functional of type $T_i$ with $i \leq 9$ is never nearly equivalent to a functional of type $T^{(r)}_{s,\lambda}$; two functionals of type $T^{(r)}_{s,\lambda}$ and $T^{(r')}_{s',\lambda'}$ with $(r,s) \neq (r', s')$ are never nearly equivalent while two functionals of type $T^{(r)}_{s,\lambda}$ and $T^{(r)}_{s,\lambda'}$ are nearly equivalent if and only if, denoted by $\mu$ and $\mu'$ a root of $p_\lambda(t)$ and $p_{\lambda'}(t)$ respectively in the algebraic closure of $\KK$, we have $\KK(\mu) = \KK(\mu')$ (see the fourth column of Table~\ref{Tab F} for the definition of $p_{\lambda}(t)$).

\subsubsection{Case $n = 6$}\label{n = 6}

Revoy~\cite{Revoy79} has classified equivalence classes of 6-dimensional trivectors. In view of his classification, when $\dim(V) = 6$ every non-trivial linear functional $h \in (\bigwedge^3V)^*$ belongs to one of the types $T_1, T_2, T_3, T_4, T^{(1)}_{10,\lambda}$ or $T^{(2)}_{10,\lambda}$ of Table~\ref{Tab F}; we remark that the latter two types
might comprise several inequivalent cases. Furthermore, in these two cases, with $p_\lambda(t)$ as in Table~\ref{Tab F}, let $\mu$ be a root of $p_\lambda(t)$ in the algebraic closure of $\KK$, let $\KK' := \KK(\mu)$ be the quadratic extension of $\KK$ by $\mu$ and $V' := \KK'\otimes_\KK V$. If $h$ has type $T^{(1)}_{10,\lambda}$ then, regarded as a linear functional of $\bigwedge^3V'$, it has type $T_3$ or $T_4$ according to whether $\chr(\KK) \neq 2$ or $\chr(\KK) = 2$. If $h$ has type $T^{(2)}_{10,\lambda}$ (whence $\chr(\KK) = 2$) then it has type $T_3$ when regarded as a linear functional of $\bigwedge^3V'$. In other words, if we replace $V$ with $V'$ then $T^{(r)}_{10,\lambda}$ disappears, absorbed by $T_3$ or $T_4$, although $\bigwedge^3V'$ can still admit linear functionals of type $T^{(r)}_{10,\lambda'}$, but with $\lambda'\in \KK'$ necessarily different from the scalar $\lambda$ previously chosen in $\KK$.

The next theorem is a recapitulation of the above, with some additional information on $H_h$. In the first two cases, the hyperplane $H_h$ is easy to describe: it is either trivial or a trivial extension of a symplectic hyperplane. In each of the remaining three cases, being currently unable to offer a nice geometric description of $H$, we only provide a description of the upper radical.

\begin{theorem}\label{case n=6}
With $\dim(V)=6$, let $h$ be a non-trivial linear functional of $\bigwedge^3V$, let $\chi$ be the alternating $3$-linear form associated to $h$ and $H := H_h$
be the hyperplane of ${\cG}_3(V)$ defined by $h$. Then one of the following occurs.
\begin{enumerate}
\item
  $h$ has type $T_1$ (rank $3$).  In this case $H$ is the trivial hyperplane centered at $\Rad(\chi)$. Its upper radical is the set of $2$-subspaces of $V$ that meet $\Rad(\chi)$ non-trivially.
\item
  $h$ has type $T_2$ (rank $5$), namely $\Rad(\chi)$ is $1$-dimensional. In this case $H$ is a trivial extension $H = E(H_0)\odot\Rad(\chi)$ of a symplectic hyperplane $E(H_0$), constructed in a complement $V'$ of $\Rad(\chi)$ in $V$ starting from the line-set $H_0$ of a symplectic generalized quadrangle living in a hyperplane of $V'$. The elements of $R^\uparrow(H)$ are the lines of $\PG(V)$ that either belong to $H_0$ or pass through the point $[\Rad(\chi)] = R_\downarrow(H)$.
\item\label{case n=6-3}
  $h$ has type $T_3$ (rank $6$).  Then $R^{\uparrow}(H) = \{\langle x, y\rangle \colon x\in V_1\setminus\{0\}, ~ y\in V_2\setminus\{0\}\}$ for a suitable decomposition $V = V_1\oplus V_2$ with $\dim(V_1) = \dim(V_2) = 3$.
\item
  $h$ has type $T_4$ (rank $6$). Then $R^{\uparrow}(H) = \{\langle x+y, \omega(x)\rangle \colon x\in V_1\setminus\{0\}, y\in V_2\}$ for a  decomposition $V = V_1\oplus V_2$ with $\dim(V_1) = \dim(V_2) = 3$ and $\omega$
  an isomorphism from $V_1$ to $V_2$.
\item
 $h$ has type $T^{(1)}_{10,\lambda}$ or $T^{(2)}_{10,\lambda}$ (rank $6$). Then $R^\uparrow(H)$ is a Desarguesian line spread of $\PG(V)$.
\end{enumerate}
\end{theorem}
\begin{proof}
In view of \cite{Revoy79}, only the claims on $H$ need to be proved. When $h$ has rank $3$ then $H$ is a trivial hyperplane by Corollary~\ref{ext hyp cor2} (recall that $R_\downarrow(H) = [\Rad(\chi)]$). Proposition~\ref{ext hyp sing} yields a description of its upper radical. Let $h$ have rank $5$.  Then, by Theorem~\ref{ext hyp th} we have $H = H'\odot \Rad(\chi)$ for a hyperplane $H'$ of ${\cG}_3(V')$ and a suitable complement $V'$ of $\Rad(\chi)$ in $V$. Moreover $R_\downarrow(H') = \emptyset$. Comparing
with the results of Theorem~\ref{l5} we see that $H'$ is a symplectic hyperplane, namely $H' = E(H_0)$ with $H_0$ as claimed. The description of $R^\uparrow(H)$ follows by claim (\ref{ext-2}) of Corollary~\ref{ext hyp cor} and what we know of $H'$.

In the remaining cases $h$ has full rank, namely $R_\downarrow(H) = \emptyset$. The information on $R^\uparrow(H)$ can be obtained by computing the kernel of $\tilde{h}$, where $\tilde{h}$ is defined as  \eqref{tildef} of Subsection~\ref{algebraic approach}. Indeed $R^\uparrow(H) = \varepsilon_2^{-1}(\ker(\tilde{h}) \cap \GG_2(V))$ by \eqref{aG} of Subsection~\ref{algebraic approach}. We refer the reader to \cite{IL16} for the details of those computations and more information on $H$.
\end{proof}

\subsubsection{Case $n = 7$ with $\KK$ a perfect field of cohomological dimension at most $1$}\label{n = 7}

Cohen and Helminck \cite{CH88} have classified linear functionals of $\bigwedge^3V$ for $n=7$ under the hypothesis that $\KK$ belongs to ${\bf Cd}\{0,1\}$ and is a perfect field, proving that any such linear functional, if non-trivial, belongs to one of types listed in Table~\ref{Tab F}
or is a scalar multiple of $T_9$.  Of course, as $\KK$ is perfect by assumption, types $T^{(1)}_{10,\lambda}$ and $T^{(1)}_{11,\lambda}$ now exist independently only when $\chr(\KK) \neq 2$.

Throughout this subsection $n = 7$. As in the previous two subsections, $h$ is a non-trivial linear functional of $\bigwedge^3V$ and $H := H_h$ is the hyperplane of $\cG_3(V)$ defined by $h$. For the moment we do not make any assumption on $\KK$, although the hypotheses of Theorem~\ref{tspr} are tailored on the classification of Cohen and Helminck \cite{CH88}. We will turn back to the hypotheses of \cite{CH88} in Corollary~\ref{n=7 conclusione1}.

Suppose firstly that $\rk(h) < 7$. Then by Corollary~\ref{ext hyp cor2} and Theorem~\ref{ext hyp th}, either $H$ is a trivial hyperplane or it is a trivial extension of a hyperplane $H_0$ of $\cG_3(V_0)$ with $R_\downarrow(H_0) = \emptyset$, for a subspace $V_0$ of $V$ of either dimension $5$ or $6$ (see also Proposition~\ref{rank 4}). All we might wish to know on $R^\uparrow(H)$ can be obtained either from Proposition~\ref{ext hyp sing} (when $H$ is trivial) or by Corollary~\ref{ext hyp cor} and the information previously achieved on hyperplanes of $\cG_3(V_0)$ (Theorems~\ref{l5} and~\ref{case n=6}).

The case $\rk(h) = 7$ is discussed in Theorem~\ref{tspr}, where  we collect some information on the geometry ${\cP}(H)$ of the poles of $H$ (see \S~\ref{poles}) under the additional assumption that $h$ belongs to one of the types of rank $7$ of Table~\ref{Tab F}, but with $\chr(\KK) \neq 2$ in case $T^{(1)}_{11,\lambda}$ (in accordance with \cite{CH88}). We recall that the lines of ${\cP}(H)$ are the elements of $R^\uparrow(H)$ and the points of ${\cP}(H)$, called poles of $H$, are just the points of $\PG(V)$ that belong to elements of $R^\uparrow(H)$. Note that, as we assumed that $h$ has full rank, $2$ and $4$ are the only possible values for the degree of a pole of $H$. As in \S~\ref{poles}, we denote by $P(H)$ the set of poles of $H$.

According to Theorem~\ref{tspr}, when $h$ has type $T_9$ the geometry ${\cP}(H)$ is a split Cayley hexagon. We refer the reader to Van Maldeghem~\cite{VM98} for a definition and the
properties of this family of generalized hexagons.

We fix some terminology and conventions which will be exploited in the statement of Theorem~\ref{tspr}.  We say that $h$ is \emph{in canonical form} if it admits a description given in Table~\ref{Tab F} with respect to the basis $E = (e_i)_{i=1}^7$ chosen for $V$. Clearly, there is no loss of generality in assuming that $h$ is in canonical form, but we will make this assumption only when necessary. For two vectors $x = \sum_{i=1}^7e_ix_i$ and $y = \sum_{i=1}^7e_iy_i$ of $V$ and $1\leq i < j \leq 7$ we put $|x, y|_{i,j} ~ := ~ x_iy_j - x_jy_i$, namely $|x,y|_{i,j}$ is the $(i,j)$-Pl\"{u}cker coordinate of $x\wedge y$ with respect to the basis $(e_i\wedge e_j)_{1\leq i< j\leq 7}$ of $V$ associated to $E$.

\begin{theorem}\label{tspr}
With $\dim(V)= 7$, let $h$ be a linear functional of $\bigwedge^3V$ belonging to one of the types of rank $7$ in Table~\ref{Tab F}, provided that $\chr(\KK) \neq 2$ for type $T^{(1)}_{11,\lambda}$. Then the following propositions hold on $H := H_h$, according to the type of $h$.
\begin{enumerate}[1)]
\item
  $h$ has type $T_5$. Two non-degenerate symplectic polar spaces ${\cS}_1$ and ${\cS}_2$ are given, with distinct hyperplanes $S_1$ and $S_2$ of $\PG(V)$ as their point-sets and such that they induce the same polar space ${\cS}_0$ on $S_0 := S_1\cap S_2$. The radical of ${\cS}_0$ is a point, say $p_0$. Two totally isotropic planes $A_1$ and $A_2$ of ${\cS}_0$ are also given in such a way that $A_1\cap A_2 = \{p_0\}$. The poles of $H$ are the points of $S_1\cup S_2$, the poles of degree 4 being the points of $A_1\cup A_2$. The lines of ${\cP}(H)$ are the totally isotropic lines of ${\cS}_i$ that meet $A_i$ non-trivially, for $i = 1, 2$.
\item
 $h$ has type $T_6$. The point-set $P(H)$ of ${\cP}(H)$ is a hyperplane of $\PG(V)$. A non-degenerate polar space $\cal S$ of symplectic type is defined over $P(H)$ and a totally isotropic plane $A$ of $\cal S$ is given. The lines of ${\cP}(H)$ are the lines of $\cal S$ that meet $A$ non-trivially. The points of $A$ are the poles of $H$ of degree $4$.
\item
  $h$ has type $T_7$. With $h$ in canonical form, let $C$ be the conic described by the equation $x_1^2 = x_2x_3$ in the plane $S := [e_1, e_2, e_3]$ of $\PG(V)$ and let $Q$ be the hyperbolic quadric of $[e_4, e_5, e_6, e_7]$ with equation $x_4x_6 + x_5x_7 = 0$. Let $\Gamma$ be the cone with $S$ as the vertex and $Q$ as the basis. Then $P(H) = \Gamma$ and $C$ is the set of poles of degree $4$. The lines of ${\cP}(H)$ are the lines of $S$ and the lines $[x,y]\subset \Gamma$ with $|x,y|_{1,4} = |x,y|_{3,7}$, $|x,y|_{1,7} = |x,y|_{2,4}$ and satisfying one of the following:
\begin{enumerate}[a)]
\item $[x,y]\cap S$ is a point of $C$ and $|x,y|_{1,5} + |x,y|_{3,6} = |x,y|_{1,6} + |x,y|_{2,5} = 0$;
\item $[x,y]\cap S = \emptyset$ and $|x,y|_{4,6} + |x,y|_{5,7} = |x,y|_{4,5} = |x,y|_{6,7} = 0$.
\end{enumerate}
\item
 $h$ has type $T_8$. In this case $H$ is a symplectic hyperplane. In view of Corollary {\rm~\ref{symp hyp cor}}, the geometry ${\cP}(H)$ is a non-degenerate polar space of symplectic type and rank $3$, naturally embedded in a hyperplane of $\PG(V)$. All poles of $H$ have degree $4$.
\item\label{tspr-5}
 $h$ has type $T_9$. Then ${\cP}(H)$ is a split Cayley hexagon naturally embedded in a non-singular quadric of $\PG(V)$. Its dual admits a natural embedding in $\cG_2(V)$. All poles of $H$ have degree $2$.
More explicitly, with $h$ in canonical form, the set $P(H)$ of the poles of $H$ is the quadric of $\PG(V)$ described by the equation $x_1x_4 + x_2x_5 + x_3x_6 = x_7^2$ and the lines of ${\cP}(H)$ are the lines $[x,y]$ of $\PG(V)$ contained in $P(H)$ and satisfying the following equations:\\
$|x,y|_{1,2} + |x,y|_{6,7} ~ = ~ |x,y|_{2,3} + |x,y|_{4,7} ~ = ~ |x,y|_{4,6} + |x,y|_{2,7} ~ = ~ 0$, \\
$|x,y|_{1,3} = |x,y|_{5,7}$, ~ $|x,y|_{4,5} = |x,y|_{3,7}$, ~ $|x,y|_{5,6} = |x,y|_{1,7}$.
\addtocounter{enumi}{1}
\item $h$ has type $T^{(r)}_{11,\lambda}$, $r = 1, 2$. With $h$ in canonical form, $P(H) = [e_2, e_3, e_5, e_6, e_7]$ and $\langle e_7\rangle$ is the unique pole of degree $4$. The elements of $R^\uparrow(H)$ are the lines $[x,y]\subset P(H)$ such that:\\
$|x,y|_{2,3} + \lambda |x,y|_{5,6} = 0$ and $|x,y|_{3,5} = |x,y|_{2,6}$ when $r = 1$;\\
$\lambda|x,y|_{2,3} + |x,y|_{2,6} + |x,y|_{3,5} = 0$ and $|x,y|_{2,3} = |x,y|_{5,6}$ for $r = 2$.
\end{enumerate}
\end{theorem}
\begin{proof}
All of the above can be proved using the technique sketched in the proof of Theorem~\ref{case n=6}, based on the investigation of $\ker(\tilde{h})$, combined with the following remark, more suited to an investigation of the poles of $H$.
 Given a non-zero vector $a\in V$, consider the degenerate alternating bilinear form $h_a(x,y) := h(a,x,y):V\times V\rightarrow \KK$. Then, the lines of $R^\uparrow(H)$ through $\langle a\rangle$ are the $2$-spaces $\langle a,x\rangle$ contained in $\Rad(h_a)$. We omit the details of these computations, referring the reader to \cite{IL16} for them. Descriptions of $P(H)$ are also given by Cohen and Helminck \cite{CH88} for all cases of Theorem~\ref{tspr} except the last one.
Observe that in all cases $[P(H)]$ consists of the points of some (possibly degenerate) quadric in $\PG(V)$; this is consistent with \cite[Theorem 3.2]{DS10}.
\end{proof}

\begin{corollary}\label{n=7 conclusione1}
Let $n = 7$ and assume that $\KK$ belongs to ${\bf Cd}\{0,1\}$ and is perfect. Then $R^\uparrow(H) \neq \emptyset$ for every hyperplane $H$ of $\cG_3(V)$.
\end{corollary}
\begin{proof} Let $h\in (\bigwedge^3V)^*$ be such that $H_h = H$. As $R^\uparrow(H)\supseteq G_2(p)$ for every $p\in R_\downarrow(H)$, if  $R_\downarrow(H) \neq \emptyset$ then $R^\uparrow(H) \neq \emptyset$. Assume that $R_\downarrow(H) = \emptyset$, namely $\rk(h) = 7$.
Then $h$ satisfies the hypotheses of Theorem~\ref{tspr}, by \cite{CH88}. Recall that $R^\uparrow(H) = \emptyset$ if and only if $P(H) = \emptyset$. However $P(H)\neq \emptyset$ in each of the cases examined in Theorem~\ref{tspr}. Hence $R^\uparrow(H) \neq \emptyset$.
\end{proof}

Henceforth, in view of the description given for ${\cP}(H)$ in Theorem~\ref{tspr}, we call a hyperplane $H=H_h$ of $\cG_3(V)$ with $h$ of type $T_9$, \emph{hexagonal}.

\begin{corollary}\label{n=7 conclusione2}
Let $n = 7$ and assume that $\KK\in {\bf Cd}\{0,1\}$ is perfect. Let $H$ be a hyperplane of $\cG_3(V)$. Then $R^\uparrow(H)$ contains no singular plane of the point-line geometry $\cG_2(V)$ if and only if $H$ is hexagonal.
\end{corollary}
\begin{proof} We know that $R^\uparrow(H)$ is a subspace of $\cG_2(V)$ (Proposition~\ref{V2(H) subspace}). This subspace contains at least a singular plane of $\cG_2(V)$ if and only if either $H$ admits at least one pole of degree $4$ or $6$ or the geometry of poles ${\cP}(H)$ contains at least one proper triangle, namely a non-collinear triple of pairwise collinear points.
Indeed, the set of lines of $\PG(V)$ through a pole of degree $6$ is a maximal singular subspace of $\cG_2(V)$ while the lines of ${\cP}(H)$ through a pole of degree $4$ form a singular subspace of $\cG_3(V)$ of rank $4$. On the other hand, the three sides of a triangle of ${\cP}(H)$, regarded as points of $\cG(V)$, span a singular plane of $\cG_2(V)$.

Let $h\in (\bigwedge^3V)^*$ be such that $H = H_h$. As $R_\downarrow(H)$ is the set of poles of $H$ of degree 6, if $\rk(h) < 7$ then $P(H)$ contains at least one maximal singular subspace of $\cG_2(V)$.  Suppose now $\rk(h) = 7$; then $h$ satisfies the hypotheses of Theorem~\ref{tspr}. In each of the cases considered in Theorem~\ref{tspr} except $T_9$, the hyperplane $H$ admits at least one pole of degree $4$. In case $T_9$ all poles have degree $2$ and  ${\cP}(H)$, being a generalized hexagon, contains no triangles.
Thus, in this case $R^\uparrow(H)$ contains no singular plane of $\cG_2(V)$.
\end{proof}

\section{Proof of Theorem~\ref{main 1}} \label{Sect4}

Throughout this section $\KK$ is a perfect field in the class ${\bf Cd}\{0,1\}$, $\dim(V) = 8$ and $H$ is a given hyperplane of $\cG_3(V)$. Moreover,  $h$ is a linear functional of $\bigwedge^3V$ associated to $H$, namely $H = \varepsilon_3^{-1}([\ker (h)])$. (Recall that $h$ is uniquely determined by $H$ up to a scalar). In the first part of this section, without assuming that $\KK$ satisfies $(*)$ of Theorem~\ref{main 1}, we prove some properties of $h$ under the assumption that $R^{\uparrow}(H)$ is a spread. In the second part of the section we shall complete the proof of Theorem~\ref{main 1} by showing that hypothesis $(*)$ of Theorem~\ref{main 1} contradicts what we have proved in the first part.

\subsection{Preliminary results}\label{Sect4 prel}

Let $V'\leq V$ be a subspace.
As stated in Subsection~\ref{Notation}, if $\dim(V') > 3$ then  $H(V') = G_3(V')\cap H$ is the set of members of $H$ contained in $V'$. Clearly, either $H(V') = G_3(V')$ or $H(V')$ is a hyperplane of $\cG_3(V')$.

We recall that, according to a definition stated in \S~\ref{n = 7}, given a $7$-dimensional vector space $V'$, the hexagonal hyperplanes of $\cG_3(V')$ are those of type $T_9$;
accordingly,  the set of poles of a hexagonal hyperplane is a non-singular quadric while its geometry of poles is a split Cayley hexagon (see case \ref{tspr-5} of Theorem~\ref{tspr}).

\begin{lemma}\label{L2-i}
The upper radical $R^\uparrow(H)$ is a spread if and only if $H(V')$ is a hexagonal hyperplane of $\cG_3(V')$, for every hyperplane $V'$ of $V$.
\end{lemma}
\begin{proof}
Suppose firstly that $R^\uparrow(H)$ is not a spread. Then $H$ admits a pole $p$ of degree $r(p)\geq  3$. Let $V'$ be a hyperplane of $V$ containing $p$ and such that $[V'/p]\cap R_p(H)$ has rank at least $3$. Such a hyperplane certainly exists. Indeed if $R_p(H) \neq [V/p]$ any hyperplane $V' < V$ containing $p$ and such that $[V'/p]\supseteq R_p(H)$ has the required property; otherwise any hyperplane of $V$ containing $p$ does the job. Clearly, $R_p(H)\cap[V'/p] \subseteq R_p(H(V'))$. Hence, either $H(V') = G_3(V')$ or $H(V')$ is a hyperplane of $\cG_3(V')$ admitting $p$ as a pole of degree greater than $2$ (whence, at least $4$). However, all poles of a hexagonal hyperplane have degree $2$ (see Theorem~\ref{tspr}, case (\ref{tspr-5})). Therefore, $H(V')$ cannot be a hexagonal hyperplane.

Conversely, assume that $R^\uparrow(H)$ is a spread and let $V'$ be a hyperplane of $V$. If $H(V')$ is not a hyperplane of $\cG_3(V')$ then $G_3(V')\subseteq H$. Given $p\in[V']$, the quotient space $[V'/p]$ is a singular subspace of the polar space ${\cS}_p(H)$ of rank $6$. However, the radical of ${\cS}_p(H)$ has rank $1$, by assumption. Hence the singular subspaces of ${\cS}_p(H)$ have rank at most $4$ --- a contradiction. It follows that $H(V')$ is a hyperplane of $\cG_3(V')$.

Let now $p\in [V']$ and let $r$ be its degree with respect to the hyperplane $H(V')$. Then $r$ is the rank of the radical $R_p(H(V'))$ of the polar space ${\cS}_p(H(V'))$. On the other hand, ${\cS}_p(H(V'))$ is the polar space induced by ${\cS}_p(H)$ on $V'/p$. As, by assumption, the radical $R_p(H)$ of ${\cS}_p(H)$ is a point of $[V/p]$, the radical $R_p(H(V'))$ of ${\cS}_p(H(V'))$ is either empty or a line of $\PG(V'/p)$, namely either $r = 0$ or $r = 2$. Thus we have proved that all poles of $H(V')$ have degree 2. It follows that the hyperplane $H(V')$ is hexagonal (compare with
Theorem~\ref{tspr} or the proof of Corollary~\ref{n=7 conclusione2}).
\end{proof}

Henceforth we assume that $R^{\uparrow}(H)$ is a line-spread of $\PG(V)$. We put $\Sigma := R^\uparrow(H)$ for short and, for a point $p\in \PG(V)$, we denote by $\ell_p$ the unique line of $\Sigma$ containing $p$. For a subspace $V'\leq V$, we put $\Sigma(V'):=\{\ell\in\Sigma~\colon~\ell\subseteq [V'] \}$.

\begin{lemma}\label{L1-i}
Assume that $\Sigma = R^{\uparrow}(H)$ is a spread and let $V'$ be a hyperplane of $V$.
Then $\ell_p\in\Sigma(V')$ for every pole $p$ of $H(V')$.
\end{lemma}
\begin{proof}
Let $p$ be a pole of $H(V')$ and, by way of contradiction, suppose that $\ell_p\not\subseteq [V']$. As $p$ is a pole of $H(V')$, we have $p\in \ell'$ for some line $\ell'\in R^\uparrow(H(V'))$. Necessarily, $\ell' \neq \ell_p$, since $\ell_p\not\subseteq [V']$. As $\ell'\in R^\uparrow(H(V'))$, every line $m\subseteq [V']$ through $p$ is collinear with $\ell'$ in ${\cS}_p(H(V'))$. Hence $m$ is collinear with $\ell'$ also in ${\cS}_p(H)$, since $H(V') \subseteq H$ and the members of $H(V')$ (respectively $H$) through $p$ are the lines of ${\cS}_p(H(V'))$ (respectively ${\cS}_p(H)$). It follows that the orthogonal space $\ell'\!^\perp$ of $\ell'$ in ${\cS}_p(H)$ contains $[V'/p]$. However, $\ell_p$ and $\ell'$ are orthogonal in ${\cS}_p(H)$, since $\ell_p$ is the radical $R_p(H)$ of ${\cS}_p(H)$. Therefore $\ell'\!^\perp = [V/p]$, namely  $\ell' \in R_p(H)$. However $R_p(H) = \{\ell_p\}$. Hence $\ell' = \ell_p$. A contradiction has been reached.
\end{proof}

\begin{note}
No use of the hypothesis that $\KK\in {\bf Cd}\{0,1\}$ and $\KK$ is perfect is made in the proof of Lemma~\ref{L1-i} while that hypothesis in exploited in the proof of Lemma~\ref{L2-i} only to claim that if all poles of $H(V')$ have degree $2$ then $H(V')$ is hexagonal. If we renounce that hypothesis then we can still prove the following weaker version of Lemma~\ref{L2-i}: $R^\uparrow(H)$ is a spread if and only if all poles of $H(V')$, if any, have degree $2$, for every hyperplane $V'$ of $V$.
\end{note}

\begin{note}
In view of Lemmas~\ref{L1-i} and~\ref{L2-i}, when $\Sigma = R^\uparrow(H)$ is a spread, the set $\Sigma(V')$ is a distance--$2$ spread of the generalized hexagon $\cP(H(V'))$. These objects are expected
to be very rare. The only examples discovered so far are defined over the field $\FF_3$, see \cite{DV04}.
\end{note}

We now turn to the linear functional $h\in (\bigwedge^3V)^*$ associated to $H$.

\begin{lemma}\label{thmn=8}
Let $R^\uparrow(H)$ be a spread. Then, for a suitable choice of the basis $E = (e_i)_{i=1}^8$ of $V$ and, possibly, up to rescaling $h,$ there exists a family $\{a_{ij}\}_{1\leq i< j \leq 6} \subseteq \KK$ of scalars such that
\begin{equation}\label{elspr}
h ~ = ~ \underline{123}+\underline{456}+\underline{147}+\underline{257}+
   \underline{367}+\sum_{1\leq i<j\leq 6} a_{ij}\cdot\underline{ij8}
 \end{equation}
and $A :=\begin{pmatrix}
    a_{14} & a_{15} & a_{16} \\
    a_{24} & a_{25} & a_{26} \\
    a_{34} & a_{35} & a_{36}
    \end{pmatrix} $ has no eigenvalue in $\KK.$
\end{lemma}
\begin{proof}
Let $V_\infty$ be a hyperplane of $V$. In view of Lemma~\ref{L2-i}, the linear functional $h_\infty$ induced by $h$ on $\bigwedge^3V_\infty$ is of type $T_9$. Hence, modulo rescaling $h$, if necessary, we can choose a basis $E_{\infty}=(e_i)_{i=1}^7$ of $V_\infty$ such that $h_\infty = \underline{123}+\underline{456}+\underline{147}+\underline{257}+ \underline{367}$ (see Table~\ref{Tab F}). We can extend $E_{\infty}$ to a basis $E = (e_i)_{i=1}^8$ of $V$. With respect to $E$, the linear functional $h$ admits the following representation:
\begin{equation}\label{elspr-prov}
h ~ = ~ \underline{123}+\underline{456}+\underline{147}+\underline{257}+
   \underline{367}+\sum_{1\leq i<j\leq 7} a_{ij}\cdot\underline{ij8}.
 \end{equation}
It remains to prove that we can choose $e_8$ in such a way that $a_{i7} = 0$ for every $i < 7$ and the matrix $A = (a_{i,3+j})_{i,j = 1}^3$ admits no eigenvalue in $\KK$.

According to the information given in case \ref{tspr-5}) of Theorem~\ref{tspr}, the point-set $Q_\infty := P(H(V_\infty))$ of the geometry of poles of $H(V_\infty)$ is the quadric of $\PG(V_\infty)$ described by the following equation: $x_7^2 = x_1x_4+x_2x_5+x_3x_6$.

Put $p_7 := [e_7]$ and let $\ell_7 := \ell_{p_7}$ be the line of $\Sigma := R^\uparrow(H)$ through $p_7$. Since $p_7\not\in Q_{\infty}$, the point $p_7$ is not a pole of $H(V_\infty)$; hence, $\ell_7$ is not contained in $\PG(V_{\infty})$. We can assume to have chosen $e_8$ in such a way that $[e_8]\in \ell_7\setminus\{p_7\}$. With this choice of $e_8$ we have $\langle e_7, e_8\rangle \in \Sigma$, hence $h(e_i\wedge e_7\wedge e_8) = 0$ for any $i < 7$, namely $a_{i7} = 0$ in \eqref{elspr-prov} for every $i < 7$. The last claim remaining to be proved is that $A$ admits no eigenvalue in $\KK$.

For any $t \in\KK$, put $V_t :=W\oplus \langle  te_7+e_8\rangle$, where $W := \langle e_i\rangle_{i=1}^6$. Then  $\{V_t\}_{t\in \KK\cup \{\infty\}}$ is the family of the hyperplanes of $V$ through $W$. By Lemma~\ref{L2-i} and the information given at case~\ref{tspr-5}) of Theorem~\ref{tspr}, the point-set $Q_t := P(H(V_t))$ of the geometry of poles of $H(V_t)$ is a non-degenerate quadric. Put $Q_{t,W}:= Q_t\cap[W]$. \\

\noindent {\it Claim \gdef\SClaim{\star} $(\SClaim)$}.
{\em With respect to the basis $(e_i)_{i=1}^6$ of $W$, the quadric $\cQ_{t,W}$ is described by the following equation:
\begin{equation}\label{claim ast 1}
\begin{pmatrix}x_1&x_2&x_3\end{pmatrix}A_t\begin{pmatrix}
     x_4\\x_5\\x_6\end{pmatrix} ~ = ~ 0,
\end{equation}
where  $A_t$ is a suitable non-singular $3\times 3$ matrix having no eigenvalue in $\KK$ for $t\neq\infty$ and $A_{\infty}=I$, where $I$ is the identity matrix of order $3$. Moreover,}
\begin{equation}\label{claim ast 2}
Q_{t,W}\cap Q_{s,W} ~= ~ \bigcup\{\ell~\colon~ \ell \in \Sigma(W)\}, ~~~~ \forall t, s\in \KK\cup\{\infty\}, ~ t\neq s.
\end{equation}
Proof of Claim $(\SClaim)$. Let $h_W$ be the linear functional induced by $h$ on $\bigwedge^3W$. Then $h_W=\underline{123}+\underline{456}$. As the planes $P := [ e_1,e_2,e_3]$ and $P' := [ e_4,e_5,e_6]$ are contained in $Q_\infty$, they are contained also in $Q_{\infty, W}$. By case~\ref{case n=6-3}) of Theorem~\ref{case n=6}, every line of $R^{\uparrow}(H(W))$ meets each of them in a point. Take $p\in P$. Then $\rk(R_p(H(W)))=3$. Consequently, the Pl\"{u}cker embedding $\varepsilon_{2,W}:\cG_2(W)\rightarrow \PG(\bigwedge^2W)$ maps $R_p(H(W))$ onto a plane $\RR_{p,W}$ of $\PG(\bigwedge^2W)$ contained in $[K(H(W))]$. By Proposition~\ref{cxx} with $t=2$, we have $\dim(K(H(W))/(K(H)\cap K(H(W)))\leq 2$. Consequently, $\RR_{p,W}\cap [K(H)] \neq \emptyset$, namely $R_p(H(W))\cap \varepsilon_2^{-1}([K(H)]) \neq\emptyset$. However $\varepsilon_2^{-1}([K(H)]) = \Sigma$ by \eqref{aG} of Subsection~\ref{algebraic approach} and $\Sigma$ does not contain any line of $\cG_2(V)$. It follows that $R_p(H(W))\cap \Sigma = \{\ell_p\}$. By case (\ref{case n=6-3}) of Theorem~\ref{case n=6}, every line $\ell\in R^{\uparrow}(H(W))$ meets both $P$ and $P'$. Therefore $\ell_p$ meets $P'$ in a point. Thus, the clause $\alpha(p):= \ell_p\cap P'$ defines a bijection $\alpha$ form $P$ to $P'$. We have $R^\uparrow(H(W)) = \Sigma(W) = \{\ell_p\}_{p\in P} = \{\ell_q\}_{q\in P'}$ and $\ell_p = \langle p, \alpha(p)\rangle$ for any $p\in P$.

Put $\Sigma_W:=\bigcup_{p\in P}\ell_p=\bigcup\{\ell\colon \ell \in \Sigma(W)\}$ and
 consider a hyperplane $V_t\neq V_{\infty}$ of $V$ containing $W$, say $V_0$. By Lemma~\ref{L1-i} the elements of $\Sigma(V_0)$ cover the quadric $Q_0\subseteq [V_0]$. Similarly, $Q_\infty$ is covered by the elements of $\Sigma(V_\infty)$. Clearly, $\Sigma_W\subseteq Q_{0,W}\cap Q_{\infty,W}$. We claim that $Q_{0,W}\cap Q_{\infty,W} = \Sigma_W$. Suppose the contrary and let $p\in(\cQ_{0,W}\cap\cQ_{\infty,W})\setminus\Sigma_W$. The line $\ell_p$ is not contained in $[W]$ (since $p\not\in\Sigma_W$) but it is contained in $[V_0\cap V_\infty]$ (by Lemma~\ref{L1-i}). It follows that $V_0=\langle W,\ell_p\rangle=V_{\infty}$. We have reached a contradiction. Therefore $Q_{0,W}\cap Q_{\infty,W} = \Sigma_W$. Similarly, $Q_{t,W}\cap Q_{s,W} = \Sigma_W$ for any choice of distinct indices $t, s \in \KK\cup\{\infty\}$, as claimed in \eqref{claim ast 2}.

Let $t\in \KK\cup\{\infty\}$. As a by-product of  \eqref{claim ast 2}, both planes $P = [e_1,e_2,e_3]$ and $P' = [e_4, e_5, e_6]$ are contained in $Q_{t,W}$. Hence $Q_{t,W}$ is described by the equation $(x_1,x_2,x_3)\cdot A_t\cdot(x_4,x_5,x_6)^T = 0$
for a suitable non--singular $3\times 3$ matrix $A_t$, as claimed in \eqref{claim ast 1}. Moreover $A_\infty = I$, as $Q_\infty$ is described by the equation
$x_7^2 = x_1x_4 + x_2x_5 + x_3x_6$. It remains to prove that $A_t$ has no eigenvalue in $\KK$ for any $t\in \KK$. Given $t\in \KK$ and $a = (a_1,a_2,a_3)$, let $p_a = [a_1e_1 + a_2e_2 + a_3e_3] \in \langle e_1, e_2, e_3\rangle$. It follows from \eqref{claim ast 2} that the point $\alpha(p_a) = P'\cap \ell_{p_a}$, regarded as a $1$-dimensional subspace of $\langle e_4, e_5, e_6\rangle$, is the complete solution of the system $a^TA_tx = a^TA_\infty x = 0$, where $x = (x_4, x_5, x_6)$ stands for the triple of coordinates of a vector of $\langle e_4, e_5, e_6\rangle$ with respect to the ordered basis $(e_4, e_5, e_6)$ of $\langle e_4, e_5, e_6\rangle$. Consequently, the system $a^TA_tx = a^TA_\infty x = 0$ has rank $2$, for any choice of $a\in \langle e_1, e_2, e_3\rangle\setminus\{0\}$. However $A_\infty = I$. Hence the system $a^TA_tx = a^Tx = 0$ has rank $2$ for any choice of $a\in \langle e_1, e_2, e_3\rangle\setminus\{0\}$. This is equivalent to say that $A_t$ admits no eigenvalues in $\KK$. Claim~$(\SClaim)$ is proved.

\medskip

We are now ready to prove that $A = (a_{i, 3+j})_{i,j=1}^3$ admits no eigenvalues in $\KK$. We shall obtain this conclusion as consequence of Claim~$(\SClaim)$, by showing that $A$ is proportional to a matrix $A_0$ associated to $Q_{0,W}$ as in $(\SClaim)$.

Given a non-zero vector $u=(u_1,u_2,u_3,u_4,u_5,u_6,u_8)\in V_0 = W\oplus\langle e_8\rangle$, let $p := \langle u\rangle$ and denote by $\pi_p$  the canonical projection of $V_0$ onto $V_0/p$ and by $\overline{\cS}_p(H(V_0)) = \pi_p^{-1}({\cS}_p(H(V_0)))$ the pre-image in $V_0$ of ${\cS}_p(H(V_0))$ by $\pi_p$. Then $\overline{\cS}_p(H(V_0))$ is a polar space of symplectic type. It is necessarily degenerate, as $p$ belongs to its radical.
The antisymmetric matrix representing $\overline{\cS}_p(H(V_0))$ with respect to the basis $E_0:= E\setminus\{e_7\}$ of $V_0$ is $M_0(u) ~:= ~ T_0(u)-T_0(u)^T$ where $T_0(u)$ is the following upper triangular matrix:

\begin{small}
\[T_0(u)=\left( \begin{array}{ccccccl}
 0 &  u_3+a_{12}u_8 & -u_2+a_{13}u_8 & a_{14}u_8 & a_{15}u_8 & a_{16}u_8 & b_1\\
& 0 & u_1+a_{23}u_8 & a_{24}u_8 & a_{25}u_8 & a_{26}u_8 & b_2 \\
& & 0 & a_{34}u_8 & a_{35}u_8 & a_{36}u_8 & b_3 \\
& & & 0 & u_6+a_{45}u_8 & -u_5+a_{46}u_8 & b_4 \\
& & & & 0 & a_{56}u_8+u_4 & b_5 \\
& & & & & 0 & b_6  \\
& & & & & & 0\\
\end{array}\right)\]
\end{small}
with
\[\begin{array}{ll}
b_1 =   -a_{12}u_2-a_{13}u_3-a_{14}u_4 -a_{15}u_5-a_{16}u_6, & b_2 = a_{12}u_1-a_{23}u_3-a_{24}u_4-a_{25}u_5-a_{26}u_6, \\
b_3 =  ~ a_{13}u_1+a_{23}u_2-a_{34}u_4-a_{35}u_5-a_{36}u_6,  & b_4 = a_{14}u_1+a_{24}u_2+a_{34}u_3-a_{45}u_5-a_{46}u_6, \\
b_5 = ~ a_{15}u_1+a_{25}u_2+a_{35}u_3+a_{45}u_4-a_{56}u_6, & b_6 =  a_{16}u_1+a_{26}u_2+a_{36}u_3+ a_{46}u_4+a_{56}u_5.
\end{array}\]
For $\bar{u}=(u_1,u_2,u_3,u_4,u_5,u_6,0)$, let $C_1,\ldots, C_7$ and $R_1 = -C_1^T,\ldots, R_7 = -C_7^T$ be the columns and the rows of $M_0(\bar{u})$. By adding the linear combination $-a_{23}C_1+a_{13}C_2-a_{12}C_3-a_{56}C_4+a_{46}C_5-a_{45}C_6$ to $C_7$ and the linear combination $-a_{23}R_1+a_{13}R_2-a_{12}R_3-a_{56}R_4+a_{46}R_5-a_{45}R_6$ to $R_7$ we obtain the following matrix:
\[M'_0(\bar{u}) ~ = ~ \left(\begin{array}{c|c}
\begin{array}{ccc|ccc}0&u_3&-u_2 &0&0&0\\
-u_3&0&u_1 &0&0&0\\
u_2&-u_1&0 &0&0&0\\
\cline{1-6}
0&0&0 &0&u_6&-u_5\\
0&0&0 &-u_6&0&u_4\\
0&0&0 &u_5&-u_4&0\\
\cline{1-6}
\end{array} & \begin{pmatrix}
 0 & -A\\
 A^T & 0\\
\end{pmatrix}\begin{pmatrix} \underline{u}_1\\
                              \underline{u}_2  \end{pmatrix}\\
\cline{1-2}
(\underline{u}_1^T,\underline{u}_2^T)\begin{pmatrix}
 0 & -A\\
 A^T & 0\\
\end{pmatrix} &
0\\
\end{array}\right)
\]
where $\underline{u}_1= (u_1, u_2, u_3)^T$, $\underline{u}_2= (u_4, u_5, u_6)^T$ and $A = (a_{i,3+j})_{i,j=1}^3$. Clearly, $M'_0(\bar{u})$ and $M_0(\bar{u})$ have the same rank. Assuming that $\bar{u} \neq 0$, let $\bar{p} = \langle \bar{u}\rangle$. We have $\bar{p} \in Q_{0, W}$ if and only if $\rk(R_{\bar{p}}(H(W)))=2$. The latter is equivalent to $\rk(M_0(\bar{u}))=4$, which in its turn is equivalent to $\rk(M'_0(\bar{u}))=4$.
It is straightforward to check that $\rk(M'_0(\bar{u}))=4$ whenever $\underline{u}_1^TA\underline{u}_2 = 0$. Hence the equation
$\underline{u}_1^TA\underline{u}_2 = 0$ implies $\bar{p}\in Q_{0,W}$, whence it implies $\underline{u}_1^TA_0\underline{u}_2 = 0$ with $A_0$ as in $(\SClaim)$. It follows that $A_0 = \lambda A$ for a scalar $\lambda\in \KK$. On the other hand, $A_0$ is not the null matrix. Therefore $\lambda\neq 0$. The proof is complete.
\end{proof}

The next lemma is useful to save some computing time when performing calculations with $h$.

\begin{lemma}\label{thmn=8-BIS}
The basis $E = (e_1)_{i=1}^8$ of $V$ considered in Lemma {\rm~\ref{thmn=8}} can be chosen in such a way as to guarantee that at least four of the entries of matrix $A$ are null. In particular, we can force $a_{25} = a_{26} = a_{34} = a_{36} = 0$.
\end{lemma}
\begin{proof} Given $E = (e_i)_{i=1}^8$ as in Lemma~\ref{thmn=8} and a $3\times 3$ matrix $C$ with determinant $\det(C) = 1$, consider the basis $E' = (e'_i)_{i=1}^8$ defined as follows:
\[(e'_1,\ldots, e'_8) ~ = ~ (e_1,\ldots, e_8)\cdot\left(\begin{array}{ccc}
C & O & O \\
O & C^{-T} & O  \\
O & O & I
\end{array}\right)\]
where $O$ stands for a suitable null matrix and $I$ is the identity matrix of
order $2$.
Let $(e'\,\!^1,\ldots, e'\,\!^8)$ be the basis of $V^*$ dual of $E'$. It is straightforward to check that $h$ admits the following expression with respect to $E'$:
\[h ~ = ~ \underline{123}+\underline{456}+\underline{147}+\underline{257}+
   \underline{367}+\sum_{1\leq i<j\leq 6} b_{ij}\cdot\underline{ij8}\]
where the symbol $\underline{ijk}$ stands now
 for $e'\,\!^i\wedge e'\,\!^j\wedge e'\,\!^k$, the coefficients $b_{ij}$ are as follows
\[\left(\begin{array}{ccc}
b_{14} & b_{15} & b_{16} \\
b_{24} & b_{25} & b_{26} \\
b_{34} & b_{35} & b_{36}
\end{array}\right) ~ = ~ C^TAC^{-T},\]
\[\left(\begin{array}{ccc}
0 & b_{12} & b_{13} \\
0 & 0 & b_{23} \\
0 & 0 & 0
\end{array}\right) ~ = ~ \fT(C^T\left(\begin{array}{ccc}
0 & a_{12} & a_{13} \\
0 & 0 & a_{23} \\
0 & 0 & 0
\end{array}\right)C), \]
\[\left(\begin{array}{ccc}
0 & b_{45} & b_{46} \\
0 & 0 & b_{56} \\
0 & 0 & 0
\end{array}\right) ~ = ~ \fT(C^{-1}\left(\begin{array}{ccc}
0 & a_{45} & a_{46} \\
0 & 0 & a_{56} \\
0 & 0 & 0
\end{array}\right) C^{-T})\]
and
$\fT(X)$ is defined for a $3\times 3$ matrix $X = (x_{ij})_{i,j=1}^3$ as
\[\fT(X) ~ = ~ \left(\begin{array}{ccc}
~ 0 ~ & x_{12}-x_{21} & x_{13}-x_{31} \\
~ 0 ~ & 0 & x_{23}-x_{32} \\
~ 0 ~ & 0 & 0
\end{array}\right).\]
We shall prove that we can always choose $C$ in such a way as
 $b_{25} = b_{26} = b_{34} = b_{36}=0$. Note firstly that at least one of $a_{16}$ and $a_{26}$ is different from $0$; otherwise $a_{36}$ would be an eigenvalue of $A$, while $A$ admits no eigenvalues.
 Modulo replacing $e_1$ with $-e_2$, $e_2$ with $e_1$ and $e_4$ with $-e_5$,
$e_5$ with $e_4$, if necessary, we can assume  $a_{16}\neq 0$. Thus, we can consider the following matrix:
\[U_1 ~ = ~ \left(\begin{array}{ccc}
1 & 0 & 0 \\
-a_{26}/a_{16} & 1 & 0 \\
-a_{36}/a_{16} & 0 & 1
\end{array}\right).\]
The matrix $A' := U_1AU_1^{-1}$ has the following form for suitable entries $a'_{ij}$:
\[A' ~ = ~ \left(\begin{array}{ccc}
a'_{14} & a'_{15} & a_{16} \\
a'_{24} & a'_{25} & 0 \\
a'_{34} & a'_{35} & 0
\end{array}\right).\]
Clearly $A'$ has no eigenvalues, since $A'$ and $A$ are similar and $A$ has no eigenvalues by assumption. Therefore $a'_{24} \neq 0$. (otherwise  $a'_{25}$ would be an eigenvalue of $A'$). So, we can consider the following matrix:
\[U_2 ~ = ~ \left(\begin{array}{ccc}
1 & a'_{25}/a'_{24} & 0 \\
0 & 1 & 0 \\
0 & -a'_{34}/a'_{24} & 1
\end{array}\right).\]
The matrix $U_2A'U_2^{-1} = U_2U_1AU_1^{-1}U_2^{-1}$ has the following form:
\[U_2A'U_2^{-1} ~ = ~  \left(\begin{array}{ccc}
b_{14} & b_{15} & b_{16} \\
b_{24} & 0 & 0 \\
0  & b_{35} & 0
\end{array}\right).\]
Therefore, if we choose $C = (U_2U_1)^T$ then we get $b_{25} = b_{26} = b_{34} = b_{36} = 0$, as we wished. \end{proof}

\begin{note}
Needless to say, with a different choice of $U_1$ and $U_2$ we can force a different quadruple of coefficients $b_{ij}$ to be null. For instance, if we replace $a'_{25}/a'_{24}$ with $-a'_{14}/a'_{24}$ in $U_2$ then we get $b_{14} = 0$ instead of $b_{25} = 0$.
\end{note}

\subsection{End of the proof.}\label{end}
Let $R^{\uparrow}(H)$ be a spread. Then, by Lemma~\ref{L2-i}, any hyperplane $H(V')$ of $\cG_3(V)$ is hexagonal, for every hyperplane $V'$ of $V$.
By the classification of Theorem~\ref{tspr}, the set of poles $P(H(V'))$  determines a non-degenerate quadric $\cQ(V')$.

We shall first show that the discriminant of the quadric $\cQ(V')$, where $V'$ is an arbitrary hyperplane of $V$, can be written as the
cube of a homogeneous polynomial.

More precisely, we will show the following. Let  $V'$ have equation  $c_1x_1+\cdots+c_8x_8=0$ with respect to the a basis of $V$ as in Lemma~\ref{thmn=8-BIS}. Then, the discriminant of $\cQ(V')$ is the cube of a homogeneous polynomial $\Delta(c_1,\ldots,c_8)$ of degree $3$ in the unknowns $c_i$, $1\leq i\leq 8.$ So, if $\KK$ satisfies the hypothesis $(*)$ of Theorem~\ref{main 1}, then $\Delta(c_1,\ldots,c_8)$ always admits at least one non-trivial solution.

Consider first the case $c_1\neq0$; then we can rewrite the equation of $V'$ as
$x_1=b_2x_2+\cdots+b_8x_8$ with $b_i=-c_i/c_1, \, 2\leq i\leq 8$.
Since in $(V')^*$ we have $e^1=b_2e^2+\cdots+b_8e^8$, the form induced by $H(V')$ on $V'$
can be written by formally replacing $\underline{1}$ with
$b_2\underline{2}+\cdots+b_8\underline{8}$ in~\eqref{elspr}. Denote this new form by $h^1$.

For any $u=(u_i)_{i=1}^8\in V'$ with $u_8\neq0$, the $8\times 8$ matrix $M^1_{V'}(u)$ representing the  alternating bilinear form
$h^1_{H,\langle u\rangle }$ (which, restricted to $V'\times V'$, is the same as $h_u$, see Section~\ref{subsec 3})
has rank at most $6$ and its last row/column is a linear combination of the remaining ones.
Also, by construction, the first row and column of $M^1_{V'}(u)$ are  null.
Thus, the matrix $M^1_{V'}(u)$ has rank $6$ if and only if the minor $\overline{M}^1_{V'}(u)$ obtained by
deleting the first and last rows and columns has non-zero determinant.

It is straightforward to see that
\[ \det\overline{M}^1_{V'}(u)=(\tilde{q}_{b_2,\ldots,b_8}(u_2,\ldots,u_7))^2 \]
with $\tilde{q}_{b_2,\ldots,b_8}(u_2,\ldots,u_7)$ a
suitable quadratic polynomial. By homogenizing the polynomial $\tilde{q}_{b_2,\ldots,b_8}(u_2,\ldots,u_7)$ in $u_8$
we obtain a quadratic form
$q_{b_2,\ldots,b_8}(u_2,\ldots,u_8)$ representing the quadric
$\cQ(V'):=P(H(V')).$ More precisely, the points of $\cQ(V')$ are the points $P=[b_2u_2+\cdots+b_8u_8,u_2\ldots,u_8]$ where $[u_2,\ldots,u_8]$ is totally singular for $q_{b_2,\ldots,b_8}$.
 The discriminant of $\cQ(V')$ is zero if, and only if, the discriminant of $q_{b_2,\ldots,b_8}$ is zero.

Let $\widetilde{\Xi^1}(b_2,\ldots,b_8)$ be the discriminant of $q_{b_2,\ldots,b_8}$.
A straightforward, yet long, computation using \cite{JS92}, see Appendix~\ref{PN},
proves that $\widetilde{\Xi^1}(b_2,\ldots,b_8)$ can be written as the cube of the polynomial
$\widetilde{\Delta^1}={\Delta}(1,b_2,\ldots,b_8)$, where ${\Delta}$ is the polyniomial of degree $3$ written in Appendix~\ref{PN} with $c_1=1$ and
$c_2,\ldots,c_8$ replaced by $b_2,\ldots,b_8$.
Let now $\Delta^1(c_1,\ldots,c_8)$ be the homogenization of $\widetilde{\Delta^1}$ where $c_i=b_ic_1$ for $2\leq i\leq 8$.

The above computations show that the points of $\PG(V^*)$ representing
the hyperplanes $V'$ of $V$ with $c_1\neq0$ where the quadric $\cQ(V')$ is singular all lie on a cubic hypersurface $\Gamma^1$ of $\PG(V^*).$

Acting in an analogous way for every component $x_i$ with $2\leq i\leq 7$, we eventually
determine a set of $7$ polynomials
$\{\Delta^i(c_1,\ldots,c_8)\}_{i=1}^7$ and corresponding hypersurfaces $\Gamma^i$.
Direct computations, performed as before, show that  for all $1\leq i\leq 7$, the polynomials $\Delta^i$ are the same; hence
they have the same set of zeroes.
It follows that $\Delta(c_1,\ldots,c_8)=0$ if, and only if, $H(V')$ is not hexagonal ($V'$ having equation $\sum_{i=1}^{8}c_ix_i=0$)---
with the possible exception of the hyperplane $V_{\infty}:x_8=0$ which has not yet been taken into account.
However, $H(V_{\infty})$ is, by the construction in Lemma~\ref{thmn=8}, hexagonal and it is immediate to see
that $\Delta(0,0,\ldots,0,1)=-1\neq0$.
\medskip

If the field $\KK$ is quasi-algebraically closed or, more in general, assumption $(*)$ holds,
then there exists at least one $(c_1,\ldots,c_8)\neq(0,0,0,0,0,0,0,0)$ with $c_i\in\KK$
such that $\Delta(c_1,\ldots,c_8)=0$. It follows now from Lemma~\ref{L2-i} that
there exists a hyperplane $V'$ of $V$ such that $H(V')$ is not hexagonal; consequently
$R^{\uparrow}(H)$ cannot be a spread.\\
\noindent The theorem is proved. {\hspace{300pt} $\Box $}

\begin{note}
 We can provide also a syntetic proof that the polynomials $\Delta^i$
have the same set of zeroes for all $1\leq i\leq 7$.
Indeed, observe that the polynomial functions $\widetilde{\Xi^1}$ formally represents
the discriminant of the quadric $q_{b_2,\ldots,b_8}$ even when the coefficients
$b_2,\ldots,b_8$ identifying the hyperplane are taken over the algebraic
closure $\overline{\KK}$ of $\KK$. As such, each $\Delta^i$ is a polynomial
function with coefficients in $\KK$ representing
an algebraic hypersurface $\Gamma^i$ of degree $3$ in $\PG(V^*\otimes\overline{\KK})$.

Observe that for $i>1$, the hypersurface  $\Gamma^i$ has in common with
$\Gamma^1$ at least the points of the set $\Gamma^i\setminus [x_1x_i=0]$ which is an open set in the Zariski topology in
$\Gamma^1$. In particular, as $\Gamma^i\cap\Gamma^1$ is closed in $\Gamma^1$,
$\Gamma^1$ must be contained in $\Gamma^i$.
As $\overline{\KK}$ is algebraically closed, this implies that $\Delta^1$ (as polynomial) must divide $\Delta^i$. Thus,
each zero of $\Delta^1$ is also a zero of $\Delta^i$.
This applies, in particular, to the zeroes of $\Delta^i$ defined
over $\KK$.

\end{note}

\begin{note}
It is possible to consider the homogeneization $\Xi^1(c_1,\ldots,c_8)$ of the polynomial $\widetilde{\Xi^1}(b_2,\ldots,b_8)$ directly.
Our computations prove that in this case
\begin{equation}
\label{ehom}
 \Xi^i(c_1,\ldots,c_8)=(\Delta(c_1,\ldots,c_8))^3\cdot c_i^{12}.
\end{equation}
We believe that this expression is not fortuitous and it would be very
interesting to investigate the geometrical reason behind such factorization.
In any case, we remark that~\eqref{ehom} by itself makes sense only
for $c_i\neq0$, as for $c_i=0$ the argument on the rank of the
matrix does not stand.
\end{note}

\begin{note}
The use of Lemma~\ref{thmn=8-BIS} provides a massive simplification in
the computations leading to the polynomial $\Delta(u_1,\ldots,u_8)$, with a reduction in the memory and time involved of approximately $10$ times.
However, we have been able also to obtain the conclusion of Section~\ref{end}
using the generic form for the matrix $A$ as provided by~\eqref{elspr} without
any simplification in the coefficients.
We choose none the less  to introduce here the more specialized form for this matrix  as in Lemma~\ref{thmn=8-BIS},
in order to be able to present a simpler polynomial in Appendix~\ref{PN}
and also to make easier to directly check the result.
\end{note}

\section{Proof of Theorem~\ref{main 2}}\label{sec 6}

Let $n = \dim(V)$ be even and let $\KK=\FF_q$ be e finite field of order $q$. Let $\psi$ be the number of point-plane flags $(p,[X])$ of $\PG(V)$ with $X\in H$ and $p\in [X]$. As $[X]$ with $X\in H$ is a projective plane of order $q$, we have $\psi=(q^2+q+1)|H|$. If $R^{\uparrow}(H)$ is a spread, then $\cS_p(H)$ is a symplectic polar space of non-degenerate rank $r:=(n-2)/2$ with $\dim(R_p(H))=1$ for any point $p\in \PG(V)$. As the lines of $\cS_p(H)$ correspond to elements of $H$ through $p$, we see that $\psi=M\frac{q^n-1}{q-1}$, where
\begin{equation}\label{M}
M ~ = ~ \frac{q^{2r}-1}{q-1}\left(1+q^2\frac{q^{2r-2}-1}{q^2-1}\right)
\end{equation}
is the number of lines of the symplectic polar space $\cS_p(H)$. So,
\begin{equation}\label{divisibility}
q^2+q+1=\frac{q^3-1}{q-1} ~~\mbox{divides} ~~ M\frac{q^n-1}{q-1}.
\end{equation}
By \eqref{M},
\begin{equation}\label{mdiv}
M\frac{q^n-1}{q-1} ~ = ~ \frac{q^{2r+2}-1}{q-1}\frac{q^{2r}-1}{q-1}\left(1+q^2\frac{q^{2r-2}-1}{q^2-1}\right).
\end{equation}
It is well known that $(q^i-1)$ divides $(q^j-1)$ if and only if $i$ divides $j$. In particular, if either $2r\equiv 1\pmod3$, that is $r\equiv 2\pmod3$, or $r\equiv 0\pmod3$ the divisibility condition \eqref{divisibility} is fulfilled. On the other hand, suppose
now $r\equiv1\pmod3$. Then,
  \begin{eqnarray}
    \frac{q^{2r+2}-1}{q-1} ~=~
    q^{2r+1}+q^{2r}+q^{2r-1}+\cdots+1~\equiv~ 1\pmod{q^2+q+1}, \\
    \frac{q^{2r}-1}{q-1}~=~q^{2r-1}+q^{2r-2}+\cdots+1~\equiv~ q+1\pmod{q^2+q+1}.
  \end{eqnarray}
Reducing \eqref{mdiv} modulus $q^2+q+1$, we get
\begin{equation}\label{div-bis}
M\frac{q^n-1}{q-1} ~ = ~ (q+1)+(q+1)q^2\frac{q^{2r-2}-1}{q-1}~\equiv ~  q+1 \pmod{q^2+q+1}.
\end{equation}
However \eqref{div-bis} contradicts \eqref{divisibility}, since $q+1\not\equiv 0 \pmod{q^2+q+1}$. Therefore, if $r\equiv1\pmod3$ then $R^\uparrow(H)$ cannot be a spread. Theorem~\ref{main 2} is proved.

\vskip.2cm\noindent
\begin{minipage}[t]{\textwidth}
Authors' addresses:
\vskip.2cm\noindent\nobreak
\centerline{
\begin{minipage}[t]{7cm}
Ilaria Cardinali and Antonio Pasini\\
Department of Information Engineering and Mathematics\\University of Siena\\
Via Roma 56, I-53100, Siena, Italy\\
ilaria.cardinali@unisi.it\\
antonio.pasini@unisi.it \\
\end{minipage}\hfill
\begin{minipage}[t]{7cm}
Luca Giuzzi\\
D.I.C.A.T.A.M. \\ Section of Mathematics \\
Universit\`a di Brescia\\
Via Branze 53, I-25123, Brescia, Italy \\
luca.giuzzi@unibs.it
\end{minipage}}
\end{minipage}
\appendix
\section{The polynomial $\Delta(u_1,\ldots,u_8)$}
\label{PN}
We present in detail the outcome of the computation of the polynomial
$\Delta(c_1,\ldots,c_8)$ of Section~\ref{Sect4}.
This result has been obtained by a straightforward implementation of
the procedure outlined in Lemma~\ref{end} using the computer algebra system \cite{JS92}.
The actual computation, in particular factoring the discriminant of the quadrics,
took, for each polynomial being considered, slightly more than $3$ hours on
a multiprocessor {\tt XEON E7540} machine and has necessitated a maximum of approximately $22$ Gb of
RAM.

\begin{tiny}
\[
\label{HugePol}
\begin{array}{l}
\Delta(c_1,\ldots,c_8):=(-a_{23}a_{24}a_{45}+a_{24}^2a_{35})c_1^3+\normalsize((-a_{23}a_{35}a_{46}+((a_{13}a_{24}+a_{14}a_{23})a_{45}-2a_{14}a_{24}a_{35}))c_2+
((a_{23}a_{24}a_{56}+\\ (a_{14}a_{23}a_{46}+
(-a_{12}a_{24}a_{45}-a_{15}a_{24}^2)))c_3+
((-a_{24}a_{45}a_{56}+(a_{35}a_{46}^2-a_{13}a_{23}a_{46}+(-a_{12}a_{23}a_{45}+(2a_{12}a_{24}a_{35}+\\
a_{15}a_{23}a_{24}))))c_4+
((-a_{23}^2a_{46}+(a_{13}a_{24}^2-a_{14}a_{23}a_{24}))c_5+
((a_{23}^2a_{45}-a_{23}a_{24}a_{35})c_6+((a_{24}a_{35}a_{46}+(-a_{13}a_{23}a_{24}+\\
a_{14}a_{23}^2))c_7+ (a_{24}a_{45}-
a_{23}^2)c_8)))))\normalsize)c_1^2+
\normalsize((-a_{23}a_{35}a_{56}+(a_{13}a_{35}a_{46}+((a_{15}a_{23}-a_{13}a_{14})a_{45}+(-a_{15}a_{24}+a_{14}^2)a_{35})))c_2^2+\\ ((-a_{13}a_{24}a_{56}+
((-a_{12}a_{35}+
(a_{15}a_{23}-a_{13}a_{14}))a_{46}+((a_{16}a_{23}+a_{12}a_{14})a_{45}+(-3a_{16}a_{24}a_{35}+a_{14}a_{15}a_{24}))))c_3+\\
(((a_{35}a_{46}+ (a_{14}a_{45}-a_{13}a_{23}))a_{56}+((a_{13}^2-a_{15}a_{45})a_{46}+(a_{16}a_{45}^2+a_{12}a_{13}a_{45}+((a_{16}a_{23}-2a_{12}a_{14})a_{35}-a_{14}a_{15}a_{23}))))c_4+\\
(((-a_{24}a_{45}-a_{23}^2)a_{56}+(a_{35}a_{46}^2+a_{13}a_{23}a_{46}+(-a_{12}a_{23}a_{45}+(a_{12}a_{24}a_{35}+(-a_{13}a_{14}a_{24}+a_{14}^2a_{23})))))c_5+\\
((2a_{13}a_{24}a_{35}-2a_{13}a_{23}a_{45})c_6+
((a_{24}a_{35}a_{56}+(-a_{14}a_{35}a_{46}+(a_{12}a_{23}a_{35}+(a_{13}^2a_{24}+(a_{15}a_{23}^2-
a_{13}a_{14}a_{23})))))c_7+\\
(a_{35}a_{46}+(2a_{13}a_{23}-a_{14}a_{45}))c_8)))))c_2+ ((a_{12}a_{24}a_{56}+((a_{16}a_{23}+a_{12}a_{14})a_{46}+ a_{14}a_{16}a_{24}))c_3^2+\\
((a_{24}a_{56}^2+
(a_{14}a_{46}+a_{12}a_{23})a_{56}+(-a_{15}a_{46}^2+
(a_{16}a_{45}-a_{12}a_{13})a_{46}+
(-a_{12}^2a_{45}+((-a_{13}a_{16}-a_{12}a_{15})a_{24}-a_{14}a_{16}a_{23}))))c_4+\\
((-2a_{12}a_{23}a_{46}-2a_{16}a_{23}a_{24})c_5+(((-a_{24}a_{45}-a_{23}^2)a_{56}+
(a_{35}a_{46}^2-a_{13}a_{23}a_{46}+(a_{12}a_{23}a_{45}+(a_{12}a_{24}a_{35}+
((2a_{15}a_{23}-\\a_{13}a_{14})a_{24}+
a_{14}^2a_{23})))))c_6+((-a_{15}a_{24}a_{46}+
(a_{16}a_{24}a_{45}+(-a_{12}a_{13}a_{24}+(a_{16}a_{23}^2+a_{12}a_{14}a_{23}))))c_7+\\
(-a_{24}a_{56}+(-a_{14}a_{46}-2a_{12}a_{23}))c_8))))c_3+
(((-a_{13}a_{46}+(-a_{12}a_{45}+ a_{15}a_{24}))a_{56}+
(-2a_{16}a_{35}a_{46}+\\ (a_{12}^2a_{35}+(a_{13}a_{16}+
a_{12}a_{15})a_{23})))c_4^2+
(((-a_{23}a_{46}-a_{14}a_{24})a_{56}+(a_{13}a_{46}^2+a_{12}a_{45}a_{46}+
(a_{16}a_{24}a_{45}+(a_{12}a_{13}a_{24}+
(a_{16}a_{23}^2-\\ a_{12}a_{14}a_{23})))))c_5+
(((a_{23}a_{45}+a_{24}a_{35})a_{56}+((-a_{13}a_{45}+2a_{14}a_{35})a_{46}+(-a_{12}a_{45}^2-
a_{15}a_{24}a_{45}+(-a_{12}a_{23}a_{35}+\\
(a_{13}^2a_{24}+(-a_{15}a_{23}^2-a_{13}a_{14}a_{23}))))))c_6+
(((-a_{13}a_{24}+2a_{14}a_{23})a_{56}+((-a_{12}a_{35}-a_{15}a_{23})a_{46}+
(a_{16}a_{23}a_{45}- 3a_{16}a_{24}a_{35})))c_7+\\ (-2a_{23}a_{56}+(a_{13}a_{46}+(a_{12}a_{45}-a_{15}a_{24})))c_8)))c_4+
((-a_{24}^2a_{56}+a_{23}a_{46}^2)c_5^2+(2(a_{24}a_{35}-\\
a_{23}a_{45})a_{46}c_6+((a_{23}a_{24}a_{56}+((a_{13}a_{24}-a_{14}a_{23})a_{46}+(a_{12}a_{24}a_{45}+a_{15}a_{24}^2)))c_7+
(3a_{23}a_{46}+a_{14}a_{24})c_8))c_5+ ((a_{23}a_{45}^2-\\
a_{24}a_{35}a_{45})c_6^2+((a_{14}a_{23}a_{45}-a_{14}a_{24}a_{35})c_7+(-3a_{23}a_{45}+3a_{24}a_{35})c_8)c_6+
((-a_{12}a_{24}a_{35}-a_{15}a_{23}a_{24})c_7^2+(a_{13}a_{24}-\\
2a_{14}a_{23})c_8c_7+2a_{23}c_8^2)))))\normalsize)c_1+
((a_{13}a_{35}a_{56}+(-a_{13}a_{15}a_{45}+(a_{16}a_{35}^2+a_{14}a_{15}a_{35})))c_2^3+
\normalsize((-a_{12}a_{35}a_{56}+(-a_{13}a_{15}a_{46}+\\
((-a_{13}a_{16}+a_{12}a_{15})a_{45}+(a_{14}a_{16}a_{35}+a_{15}^2a_{24}))))c_3+((a_{13}^2a_{56}+((a_{13}a_{16}-a_{12}a_{15})a_{35}-a_{15}^2a_{23}))c_4+
(((a_{35}a_{46}+\\
(a_{14}a_{45}+a_{13}a_{23}))a_{56}+(-a_{15}a_{45}a_{46}+(a_{16}a_{45}^2+a_{12}a_{13}a_{45}+((2a_{16}a_{23}-a_{12}a_{14})a_{35}+(-a_{13}a_{15}a_{24}+a_{14}a_{15}a_{23})))))c_5+\\
((a_{13}^2a_{45}+(-a_{12}a_{35}^2+(-a_{15}a_{23}-a_{13}a_{14})a_{35}))c_6+((-a_{14}a_{35}a_{56}+(-a_{16}a_{35}a_{45}+(-a_{12}a_{13}a_{35}-a_{13}a_{15}a_{23})))c_7+
(a_{35}a_{56}+\\
(-a_{15}a_{45}-a_{13}^2))c_8))))\normalsize)c_2^2+
\normalsize(((-a_{13}a_{16}+a_{12}a_{15})a_{46}+(a_{12}a_{16}a_{45}+ 2a_{15}a_{16}a_{24}))c_3^2+
((-2a_{12}a_{13}a_{56}+(-2a_{12}a_{16}a_{35}-\\
2a_{15}a_{16}a_{23}))c_4+
((a_{24}a_{56}^2+(a_{14}a_{46}-a_{12}a_{23})a_{56}+(-a_{15}a_{46}^2+
(a_{16}a_{45}+a_{12}a_{13})a_{46}+(-a_{12}^2a_{45}+((-a_{13}a_{16}-a_{12}a_{15})a_{24}+\\
a_{14}a_{16}a_{23}))))c_5+
(((a_{35}a_{46}+(a_{14}a_{45}+a_{13}a_{23}))a_{56}+((-a_{15}a_{45}+a_{13}^2)a_{46}+(a_{16}a_{45}^2-a_{12}a_{13}a_{45}+(a_{16}a_{23}a_{35}+(-2a_{13}a_{15}a_{24}+\\
a_{14}a_{15}a_{23})))))c_6+
((-a_{15}a_{24}a_{56}+(-a_{16}a_{35}a_{46}+(a_{12}^2a_{35}+(-a_{13}a_{16}+a_{12}a_{15})a_{23})))c_7+(-a_{15}a_{46}+(-a_{16}a_{45}+2a_{12}a_{13}))c_8))))c_3+\\
((-a_{13}a_{56}^2+
(-a_{16}a_{35}-a_{14}a_{15})a_{56}+(a_{15}^2a_{46}-a_{15}a_{16}a_{45}))c_4^2+((-a_{23}a_{56}^2+(a_{13}a_{46}+(-a_{12}a_{45}+a_{14}^2))a_{56}+((-a_{16}a_{35}-\\ a_{14}a_{15})a_{46}+
(a_{14}a_{16}a_{45}+(a_{12}^2a_{35}+(a_{13}a_{16}+a_{12}a_{15})a_{23}))))c_5+((-2a_{13}a_{45}a_{56}+(2a_{15}a_{35}a_{46}-2a_{16}a_{35}a_{45}))c_6+(((a_{15}a_{23}- \\
a_{13}a_{14})a_{56}+
(a_{13}a_{15}a_{46}+
(a_{12}a_{15}a_{45}+(-a_{14}a_{16}a_{35}+a_{15}^2a_{24}))))c_7+(3a_{13}a_{56}+(3a_{16}a_{35}+a_{14}a_{15}))c_8)))c_4+(((a_{23}a_{46}+a_{14}a_{24})a_{56}+\\
((a_{12}a_{45}-a_{15}a_{24})a_{46}+(2a_{16}a_{24}a_{45}+(a_{12}a_{13}a_{24}+(a_{16}a_{23}^2-a_{12}a_{14}a_{23})))))c_5^2+(((-a_{23}a_{45}+a_{24}a_{35})a_{56}+(a_{13}a_{45}a_{46}+\\
(-a_{12}a_{45}^2-a_{15}a_{24}a_{45}+(-a_{12}a_{23}a_{35}+(a_{13}^2a_{24}+(-a_{15}a_{23}^2-a_{13}a_{14}a_{23}))))))c_6+((-a_{13}a_{24}a_{56}+((-a_{12}a_{35}-a_{15}a_{23})a_{46}+\\
((a_{16}a_{23}-a_{12}a_{14})a_{45}+(-3a_{16}a_{24}a_{35}-a_{14}a_{15}a_{24}))))c_7+(a_{23}a_{56}+(-2a_{13}a_{46}+(a_{12}a_{45}+(-a_{15}a_{24}-a_{14}^2))))c_8))c_5+\\
((a_{35}^2a_{46}+(-a_{13}a_{45}^2+a_{14}a_{35}a_{45}))c_6^2+((a_{23}a_{35}a_{56}+(a_{13}a_{35}a_{46}+((a_{12}a_{35}-a_{13}a_{14})a_{45}+(a_{15}a_{24}+a_{14}^2)a_{35})))c_7+
(3a_{13}a_{45}-\\2a_{14}a_{35})c_8)c_6+(((-a_{16}a_{23}+a_{12}a_{14})a_{35}+a_{13}a_{15}a_{24})c_7^2+(-a_{12}a_{35}+(-a_{15}a_{23}+a_{13}a_{14}))c_8c_7-2a_{13}c_8^2))))\normalsize)c_2+
((a_{12}a_{16}a_{46}+\\ a_{16}^2a_{24})c_3^3+\normalsize((a_{12}^2a_{56}-a_{16}^2a_{23})c_4+((-a_{12}^2a_{46}-a_{12}a_{16}a_{24})c_5+((a_{24}a_{56}^2+(a_{14}a_{46}-a_{12}a_{23})a_{56}+(-a_{15}a_{46}^2+
(a_{16}a_{45}-\\
a_{12}a_{13})a_{46}+(-2a_{13}a_{16}a_{24}+a_{14}a_{16}a_{23})))c_6+((-a_{16}a_{24}a_{56}+a_{12}a_{16}a_{23})c_7+(-a_{16}a_{46}-a_{12}^2)c_8)))\normalsize)c_3^2+
\normalsize((a_{12}a_{56}^2- a_{14}a_{16}a_{56}+\\ (a_{15}a_{16}a_{46}-a_{16}^2a_{45}))c_4^2+((-2a_{12}a_{46}-2a_{16}a_{24})a_{56}c_5+((-a_{23}a_{56}^2+(-a_{13}a_{46}+(a_{12}a_{45}+
(2a_{15}a_{24}+a_{14}^2)))a_{56}+
((-a_{16}a_{35}-\\a_{14}a_{15})a_{46}+(a_{14}a_{16}a_{45}+(a_{12}^2a_{35}+(a_{13}a_{16}+a_{12}a_{15})a_{23}))))c_6+(((a_{16}a_{23}+a_{12}a_{14})a_{56}+
(a_{13}a_{16}a_{46}+(a_{12}a_{16}a_{45}+
a_{15}a_{16}a_{24})))c_7+\\(-3a_{12}a_{56}+a_{14}a_{16})c_8)))c_4+
((a_{12}a_{46}^2+a_{16}a_{24}a_{46})c_5^2+(((a_{23}a_{46}+a_{14}a_{24})a_{56}+
(a_{13}a_{46}^2+(-a_{12}a_{45}-2a_{15}a_{24})a_{46}+
(a_{16}a_{24}a_{45}+\\(a_{12}a_{13}a_{24}+
(a_{16}a_{23}^2-a_{12}a_{14}a_{23})))))c_6+((-a_{12}a_{14}a_{46}-a_{14}a_{16}a_{24})c_7+
(3a_{12}a_{46}+3a_{16}a_{24})c_8))c_5+(((-a_{23}a_{45}+ 2a_{24}a_{35})a_{56}+\\
((-a_{13}a_{45}+a_{14}a_{35})a_{46}+(-a_{12}a_{23}a_{35}+(a_{13}^2a_{24}+(-a_{15}a_{23}^2-a_{13}a_{14}a_{23})))))c_6^2+((-a_{13}a_{24}a_{56}+
((-a_{12}a_{35}-
(a_{15}a_{23}+a_{13}a_{14}))a_{46}+\\
(a_{16}a_{23}a_{45}-3a_{16}a_{24}a_{35})))c_7+
(a_{23}a_{56}+(a_{13}a_{46}+(-2a_{12}a_{45}+(-4a_{15}a_{24}-a_{14}^2))))c_8)c_6+
(a_{13}a_{16}a_{24}c_7^2-
(a_{16}a_{23}+a_{12}a_{14})c_8c_7+\\ 2a_{12}c_8^2)))\normalsize)c_3+
(((a_{13}a_{16}+a_{12}a_{15})a_{56}+a_{16}^2a_{35})c_4^3+\normalsize(((a_{16}a_{23}-a_{12}a_{14})a_{56}+
((-a_{13}a_{16}-
a_{12}a_{15})a_{46}-
a_{15}a_{16}a_{24}))c_5+
(((a_{12}a_{35}-\\(a_{15}a_{23}+a_{13}a_{14}))a_{56}+((a_{13}a_{16}+a_{12}a_{15})a_{45}+(-2a_{14}a_{16}a_{35}+a_{15}^2a_{24})))c_6+
((a_{14}a_{56}^2+
(a_{16}a_{45}-a_{15}a_{46})a_{56}-
a_{12}a_{16}a_{35})c_7+\\ (-a_{56}^2+(-a_{13}a_{16}-a_{12}a_{15}))c_8))\normalsize)c_4^2+
\normalsize((-a_{12}a_{24}a_{56}+((-a_{16}a_{23}+a_{12}a_{14})a_{46}+
a_{14}a_{16}a_{24}))c_5^2+
((-a_{13}a_{24}a_{56}+((-a_{12}a_{35}+\\
(a_{15}a_{23}+a_{13}a_{14}))a_{46}+((a_{16}a_{23}-a_{12}a_{14})a_{45}+
(-3a_{16}a_{24}a_{35}-a_{14}a_{15}a_{24}))))c_6+
((a_{24}a_{56}^2-
a_{14}a_{46}a_{56}+(a_{15}a_{46}^2-
a_{16}a_{45}a_{46}-\\a_{13}a_{16}a_{24}))c_7+(2a_{46}a_{56}+(-a_{16}a_{23}+a_{12}a_{14}))c_8))c_5+
((-a_{23}a_{35}a_{56}+ ((a_{12}a_{35}-(a_{15}a_{23}+a_{13}a_{14}))a_{45}+(2a_{15}a_{24}+
a_{14}^2)a_{35}))c_6^2+\\
(((-a_{35}a_{46}+a_{14}a_{45})a_{56}+(-a_{15}a_{45}a_{46}+(a_{16}a_{45}^2+
((a_{16}a_{23}+
a_{12}a_{14})a_{35}+
a_{13}a_{15}a_{24}))))c_7+(-2a_{45}a_{56}+(-a_{12}a_{35}+
(a_{15}a_{23}+\\a_{13}a_{14})))c_8)c_6+((-a_{15}a_{24}a_{56}+a_{16}a_{35}a_{46})c_7^2+
(-2a_{14}a_{56}+(a_{15}a_{46}-a_{16}a_{45}))c_8c_7+2a_{56}c_8^2))\normalsize)c_4+
((a_{12}a_{24}a_{46}+a_{16}a_{24}^2)c_5^3+\\
\normalsize((a_{13}a_{24}a_{46}-(a_{12}a_{24}a_{45}+a_{15}a_{24}^2))c_6-
((a_{24}a_{46}a_{56}+a_{16}a_{23}a_{24})c_7+(a_{12}a_{24}-a_{46}^2)c_8)\normalsize)c_5^2+
\normalsize((a_{23}a_{35}a_{46}+
(-a_{13}a_{24}a_{45}+\\
a_{14}a_{24}a_{35}))c_6^2+
((a_{24}a_{45}a_{56}+(a_{35}a_{46}^2+(a_{12}a_{24}a_{35}+a_{15}a_{23}a_{24})))c_7+
(2a_{45}a_{46}+a_{13}a_{24})c_8)c_6+
((a_{15}a_{24}a_{46}- a_{16}a_{24}a_{45})c_7^2+\\
(-a_{24}a_{56}+a_{14}a_{46})c_8c_7-2a_{46}c_8^2)\normalsize)c_5+
((-a_{23}a_{35}a_{45}+a_{24}a_{35}^2)c_6^3+\normalsize((-a_{35}a_{45}a_{46}+
(a_{13}a_{24}-a_{14}a_{23})a_{35})c_7+
(-a_{45}^2+\\a_{23}a_{35})c_8\normalsize)c_6^2+
\normalsize((-a_{24}a_{35}a_{56}-a_{14}a_{35}a_{46})c_7^2+(a_{35}a_{46}-a_{14}a_{45})c_8c_7+2a_{45}c_8^2\normalsize)c_6+
(a_{16}a_{24}a_{35}c_7^3+a_{15}a_{24}c_8c_7^2+a_{14}c_8^2c_7-c_8^3))))))
\end{array} \]
\end{tiny}


\begin{thebibliography}{999}



\bib{Bou}{book}{
   author={Bourbaki, N.},
   title={\'El\'ements de math\'ematique. Alg\`ebre. Chapitres 1 \`a 3},
   language={French},
   publisher={Hermann, Paris},
   date={1970},
   pages={xiii+635 pp. (not consecutively paged)},
}
		



\bib{VCP}{article}{
   author={Brown, R. B.},
   author={Gray, A.},
   title={Vector cross products},
   journal={Comment. Math. Helv.},
   volume={42},
   date={1967},
   pages={222--236},
   issn={0010-2571},
}



\bib{BC}{book}{
   author={Buekenhout, F.},
   author={Cohen, A. M.},
   title={Diagram geometry},
   series={Ergebnisse der Mathematik und ihrer Grenzgebiete. 3. Folge. A
   Series of Modern Surveys in Mathematics [Results in Mathematics and
   Related Areas. 3rd Series. A Series of Modern Surveys in Mathematics]},
   volume={57},
   publisher={Springer, Heidelberg},
   date={2013},
   pages={xiv+592},
   isbn={978-3-642-34452-7},
   isbn={978-3-642-34453-4},
   doi={10.1007/978-3-642-34453-4},
}
		
\bib{IL16}{report}{
  author={Cardinali, I.},
  author={Giuzzi, L.},
  title={Geometries arising from trilinear forms on low-dimensional vector spaces},
  date={in preparation}
}



\bib{CH88}{article}{
   author={Cohen, A. M.},
   author={Helminck, A. G.},
   title={Trilinear alternating forms on a vector space of dimension $7$},
   journal={Comm. Algebra},
   volume={16},
   date={1988},
   number={1},
   pages={1--25},
   issn={0092-7872},
   doi={10.1080/00927878808823558},
}

\bib{co}{article}{
   author={Cooperstein, B. N.},
   title={External flats to varieties in $\PG(\wedge^2(V))$ over
   finite fields},
   journal={Geom. Dedicata},
   volume={69},
   date={1998},
   number={3},
   pages={223--235},
   issn={0046-5755},
   doi={10.1023/A:1005053409486},
}

\bib{DV04}{article}{
   author={De Wispelaere, A.},
   author={Van Maldeghem, H.},
   title={A distance-2-spread of the generalized hexagon H(3)},
   journal={Ann. Comb.},
   volume={8},
   date={2004},
   number={2},
   pages={133--154},
   issn={0218-0006},
   doi={10.1007/s00026-004-0211-9},
}

\bib{Djok}{article}{
   author={Djokovi\'{c}, D. {\u{Z}}.},
   title={Classification of trivectors of an eight-dimensional real vector space},
   journal={Lin. Multilin. Alg.},
   volume={13},
   date={1988},
   pages={3--39},
}

\bib{CAG12}{book}{
   author={Dolgachev, I. V.},
   title={Classical algebraic geometry},
   publisher={Cambridge University Press, Cambridge},
   date={2012},
   pages={xii+639},
   isbn={978-1-107-01765-8},
   doi={10.1017/CBO9781139084437},
}


\bib{DS10}{article}{
   author={Draisma, J.},
   author={Shaw, R.},
   title={Singular lines of trilinear forms},
   journal={Linear Algebra Appl.},
   volume={433},
   date={2010},
   number={3},
   pages={690--697},
   issn={0024-3795},
   doi={10.1016/j.laa.2010.03.040},
}


\bib{DS14}{article}{
   author={Draisma, J.},
   author={Shaw, R.},
   title={Some noteworthy alternating trilinear forms},
   journal={J. Geom.},
   volume={105},
   date={2014},
   number={1},
   pages={167--176},
   issn={0047-2468},
   doi={10.1007/s00022-013-0202-2},
}


\bib{C1}{book}{
   author={Gille, P.},
   author={Szamuely, T.},
   title={Central simple algebras and Galois cohomology},
   series={Cambridge Studies in Advanced Mathematics},
   volume={101},
   publisher={Cambridge University Press, Cambridge},
   date={2006},
   pages={xii+343},
   isbn={978-0-521-86103-8},
   isbn={0-521-86103-9},
   doi={10.1017/CBO9780511607219},
}

\bib{Gu64}{book}{
   author={Gurevich, G. B.},
   title={Foundations of the theory of algebraic invariants},
   series={Translated by J. R. M. Radok and A. J. M. Spencer},
   publisher={P. Noordhoff Ltd., Groningen},
   date={1964},
   pages={viii+429},
}

\bib{hall-shult93}{article}{
   author={Hall, J. I.}
   author={Shult, E. E.},
   title={Geometric hyperplanes of nonembeddable Grassmannians. },
   journal={European J. Combin. },
   volume={14},
   date={1993},
   number={1},
   pages={29--35},
}

\bib{harris}{book}{
   author={Harris, J.},
   title={Algebraic geometry},
   series={Graduate Texts in Mathematics},
   volume={133},
   publisher={Springer-Verlag, New York},
   date={1995},
   pages={xx+328},
   isbn={0-387-97716-3},
}


\bib{JS92}{book}{
   author={Jenks, Richard D.},
   author={Sutor, Robert S.},
   title={AXIOM},
   publisher={Numerical Algorithms Group, Ltd., Oxford; Springer-Verlag, New
   York},
   date={1992},
   pages={xxiv+742},
   isbn={0-387-97855-0},
}

\bib{Lou}{book}{
   author={Lounesto, P.},
   title={Clifford algebras and spinors},
   series={London Mathematical Society Lecture Note Series},
   volume={286},
   edition={2},
   publisher={Cambridge University Press, Cambridge},
   date={2001},
   pages={x+338},
   isbn={0-521-00551-5},
   doi={10.1017/CBO9780511526022},
}

%


\bib{Revoy79}{article}{
   author={Revoy, P.},
   title={Trivecteurs de rang $6$},
   language={French},
   journal={Bull. Soc. Math. France M\'em.},
   number={59},
   date={1979},
   pages={141--155},
   issn={0583-8665},
}

\bib{Serre}{book}{
  author={Serre, J.},
  title={Galois cohomology},
  series={Springer Monographs in Mathematics},
  publisher={Springer, Berlin},
  date={1997},
  pages={x+210},
  isbn={3-540-61990-9},
}



\bib{shult92}{article}{
   author={Shult, E. E.},
   title={Geometric hyperplanes of embeddable Grassmannians. },
   journal={J. Algebra},
   volume={145},
   date={1992},
   number={1},
   pages={55--82},
}

\bib{shult11}{book}{
   author={Shult, E. E.},
   title={Points and lines},
   series={Universitext},
   publisher={Springer, Heidelberg},
   date={2011},
   pages={xxii+676},
   isbn={978-3-642-15626-7},
   doi={10.1007/978-3-642-15627-4},
}

\bib{VM98}{book}{
   author={Van Maldeghem, H.},
   title={Generalized polygons},
   series={Monographs in Mathematics},
   volume={93},
   publisher={Birkh\"auser Verlag, Basel},
   date={1998},
   pages={xvi+502},
   isbn={3-7643-5864-5},
   doi={10.1007/978-3-0348-0271-0},
}

\bib{We}{article}{
   author={Westwick, R.},
   title={Real trivectors of rank seven},
   journal={Linear and Multilinear Algebra},
   volume={10},
   date={1981},
   number={3},
   pages={183--204},
   issn={0308-1087},
   doi={10.1080/03081088108817411},
}

\end{thebibliography}
 \end{document}